\documentclass[11pt]{article}
\usepackage{fullpage}
\usepackage{amssymb,latexsym,amsmath,amsthm}     
\usepackage{bbm}
\usepackage{eqnarray}
\usepackage[utf8]{inputenc}
\usepackage{enumerate}
\usepackage{enumitem}
\usepackage[english]{babel}
\usepackage{color}
\usepackage[dvipsnames]{xcolor}
\usepackage[toc,page]{appendix}
\usepackage{scalerel,stackengine}
\usepackage{titling}
\usepackage{mathrsfs}
\usepackage{float}
\usepackage{marginnote}
\usepackage[leftcaption]{sidecap}
\usepackage{hyperref}
\hypersetup{
    colorlinks=true,
    linkcolor=blue,
    filecolor=blue,      
    urlcolor=blue,
    citecolor=blue,
    linktocpage=true
}
\usepackage{mathrsfs}
\usepackage{graphicx}
\usepackage{caption}
\usepackage{color}
\usepackage{enumerate}
\usepackage{esint}
\usepackage{fancyhdr}
\usepackage{ulem}
\usepackage{scalerel}
\makeatletter
\theoremstyle{plain}
\newtheorem{thm}{Theorem}[section]
\newtheorem{prop}[thm]{Proposition}
\newtheorem{defn}[thm]{Definition}
\newtheorem{lemma}[thm]{Lemma}

\newtheorem{assumption}[thm]{Assumption}
\theoremstyle{remark}
\newtheorem{remark}[thm]{Remark}

\newcommand{\eps}{{\varepsilon}} 
\newcommand{\e}{_{\varepsilon}} 
\newcommand{\R}{\mathbb{R}}
\newcommand{\N}{\mathbb{N}}

\newcommand{\cE}{{\mathcal{E}}}
\newcommand{\cB}{{\mathcal{B}}}
\newcommand{\sB}{{\mathscr{B}}}
\newcommand{\sD}{{\mathscr D}}
\newcommand{\sA}{{\mathscr A}}
\newcommand{\sM}{{\mathscr M}}
\newcommand{\tsq}[1]{{\stackrel{2}{\rightharpoonup}_{#1}}}
\newcommand{\ts}{{\stackrel{2}{\rightharpoonup}}}
\newcommand{\weakto}{\rightharpoonup}
\newcommand{\tomega}{\tilde{\omega}}
\newcommand{\CL}{{\mathscr{C\!P}}}
 
\newcommand{\pinv}{P_{{\mathrm{inv}}}} 
\newcommand{\ex}[1]{\left\langle {#1} \right\rangle} 
\newcommand{\unf}{\mathcal{T}_{\varepsilon}} 
\newcommand{\brac}[1]{\left({#1}\right) } 
\newcommand{\cb}[1]{\left\lbrace {#1} \right\rbrace}
\newcommand{\pot}{{\mathrm{pot}}}
\newcommand{\inv}{{\mathrm{inv}}}
\newcommand{\ns}[1]{\|{#1}\|_{L^p(\Omega\times Q)}} 
\newcommand{\ltp}{L^p(\Omega \times Q)} 
\newcommand{\ltq}{L^q(\Omega \times Q)}
\newcommand{\wt}{\overset{2}{\rightharpoonup}} 
\newcommand{\st}{\overset{2}{\rightarrow}} 
\newcommand{\re}[1]{\mathbb{R}^{#1}}
\newcommand{\sob}{W^{1,p}}

\makeatother

\usepackage{authblk}
\title{Stochastic two-scale convergence and Young measures}
\author[1,2]{Martin Heida\thanks{martin.heida@tum.de, martin.heida@wias-berlin.de}} 
\author[3]{Stefan Neukamm\thanks{stefan.neukamm@tu-dresden.de}}
\author[3]{Mario Varga\thanks{mario.varga@tu-dresden.de}}
\affil[1]{Fakult{\"a}t f{\"u}r Mathematik, Technische Universit{\"a}t M{\"u}nchen}
\affil[2]{Weierstrass Institute for Applied Analysis and Stochastics, Berlin}
\affil[3]{Fakultät Mathematik, Technische Universit\"at Dresden}
\begin{document}
\maketitle
\begin{abstract}
In this paper we compare the notion of stochastic two-scale convergence in the mean (by Bourgeat, Mikeli{\'c} and Wright), the notion of stochastic unfolding (recently introduced by the authors), and the quenched notion of stochastic two-scale convergence (by Zhikov and Pyatnitskii). In particular, we introduce stochastic two-scale Young measures as a tool to compare mean and quenched limits. Moreover, we discuss two examples, which can be naturally analyzed via stochastic unfolding, but which cannot be treated via quenched stochastic two-scale convergence.
\medskip

\end{abstract}

\setlength{\parindent}{0pt}
\tableofcontents

\section{Introduction}\label{Intro}

In this paper we compare \textit{quenched} stochastic two-scale convergence \cite{Zhikov2006} with the notion of \textit{stochastic unfolding} \cite{neukamm2017stochastic,heida2019stochastic}, which is equivalent to stochastic two-scale convergence in the \textit{mean} \cite{bourgeat1994stochastic}. In particular, we introduce the concept of stochastic two-scale Young measures to relate quenched stochastic two-scale limits with the mean limit  and discuss examples of convex homogenization problems that can be treated with two-scale convergence in the mean, but not conveniently in the quenched setting of two-scale convergence.
\medskip

Two-scale convergence has been introduced in \cite{nguetseng1989general,allaire1992homogenization,lukkassen2002two} for homogenization problems (partial differential equations or variational problems) with periodic coefficients. The essence of two-scale convergence is that the two-scale limit of an oscillatory sequence captures oscillations that emerge along the sequence and that are to leading order periodic on a definite microscale, typically denoted by $\varepsilon>0$. It is especially well-suited for problems where oscillations of solutions solely stem from prescribed oscillations of the coefficients or the data. For instance, this is the case for equations with a \textit{monotone} structure or \textit{convex} variational problems. In contrast, problems that feature pattern formation to leading order (e.g., nonconvex variational problems or singular partial differential equations with non-convex domain) typically cannot be conveniently treated with two-scale convergence. Another well established method for periodic homogenization is \textit{periodic unfolding}, see \cite{cioranescu2002periodic,Visintin2006,mielke2007two,cioranescu2008periodic} as well as \cite{vogt1980, arbogast1990derivation} for the periodic modulation method, which is related. These methods build on an isometric operator---the periodic unfolding (or dilation) operator. It allows us to embed  oscillatory sequences into a larger two-scale space and to transform an oscillatory problem into an ``unfolded'' problem on the two-scale space. The latter often features a better separation of macro- and microscopic properties, which often is convenient for the analysis. We refer to \cite{griso2004error,cioranescu2004homogenization,neukamm2010homogenization,cioranescu2012periodic,hanke2011homogenization,liero2015homogenization,mielke2014two} for various interesting applications of this method. Both notions are closely linked, since weak convergence of ``unfolded'' sequence in the two-scale space is equivalent to weak two-scale convergence, see \cite{bourgeat1994rigorous}.
\medskip

In this paper we are interested in stochastic homogenization, i.e.~problems with random coefficients with a stationary distribution. The first stochastic homogenization result has been obtained by Papanicolaou and Varadhan in \cite{Papanicolaou1979} (and independently by Kozlov \cite{Kozlov1979}) for linear, elliptic equations with stationary and ergodic random coefficients on $\R^d$. In their seminal paper, Papanicolaou and Varadhan introduce a functional analytic framework, which, by now, is the standard way to model random coefficients. We briefly recall it in the special case of convex integral functionals with quadratic growth: Let $(\Omega,\mathcal F,P)$ denote a probability space of parameter fields $\omega\in\Omega$ and let $\tau_x:\Omega\to\Omega$, $x\in\R^d$, denote a measure  preserving and ergodic group action, see Assumption~\ref{Assumption_2_1} for details. A standard model for a convex, integral functional with a stationary, ergodic, random microstructure on scale $\varepsilon>0$ is then given by the functional $\mathcal{E}_{\varepsilon}^{\omega}: H^1(Q) \to \R\cup \cb{\infty}$,
\begin{equation*}
\mathcal{E}_{\varepsilon}^{\omega}(u) = \int_{Q}V\brac{\tau_{\frac{x}{\varepsilon}}\omega, \nabla u(x)}-f(x)u(x)\,dx
\end{equation*}
where $Q\subset\R^d$ denotes an open and bounded domain, $f\in L^2(Q)$, and $V(\omega,F)$ is an integrand that is measurable in $\omega\in\Omega$, convex in $F\in\R^d$, and satisfies a quadratic growth condition. A classical result \cite{DalMaso1986} shows that in the homogenization limit $\varepsilon\to 0$, the functionals $\Gamma$-converge to the homogenized functional $\mathcal{E}_{\hom}:H^1(Q) \to \R\cup \cb{\infty}$, given by
\begin{equation*}
\mathcal{E}_{\hom}(u) = \int_{Q}V_{\hom}(\nabla u(x))-f(x)u(x)\,dx,
\end{equation*}
where $V_{\hom}$ is a deterministic, convex integrand and characterized by a homogenization formula, see \eqref{equation} below.
There are different natural choices for the topology when passing to this limit:
\begin{itemize}
\item In the \textit{mean} setting, minimizers $u_\varepsilon^\omega$  of $\mathcal E_{\varepsilon}^\omega$, $\omega\in\Omega$, are viewed as random fields $(\omega,x)\mapsto u_\varepsilon^\omega(x)$ in $L^2(\Omega; H^{1}(Q))$ and one considers $\Gamma$-convergence of the averaged functional $L^2(\Omega;H^{1}(Q))\ni u\mapsto \int_\Omega\mathcal E_{\varepsilon}(u)\,dP$ w.r.t.~strong convergence in $L^2(\Omega\times Q)$.  In fact, the first result in stochastic homogenization  \cite{Papanicolaou1979} establishes convergence of solutions in this mean sense.
\item In the \textit{quenched} setting, one studies the limiting behavior of a minimizer $u_\varepsilon\in H^{1}(Q)$ of $\mathcal E_{\varepsilon}^\omega$ for fixed $\omega\in\Omega$. One then considers $\Gamma$-convergence of $\mathcal E_{\varepsilon}^\omega$ w.r.t.~strong convergence in $L^2(Q)$ for $\mathbb P$-a.a.~$\omega\in\Omega$.  
\end{itemize}
Similarly, two variants of stochastic two-scale convergence have been introduced as generalizations of periodic two-scale convergence (for the sake of brevity, we restrict the following review to the Hilbert-space case $p=2$, and note that the following extends to $L^p(\Omega \times Q)$ with $p\in (1,\infty)$):
\begin{itemize}
\item In \cite{bourgeat1994stochastic, andrews1998stochastic} the \textit{mean} variant has been introduced as follows: We say that a sequence of random fields $(u\e)\subset L^2(\Omega\times Q)$ \textit{stochastically two-scale converges in the mean} to $u\in L^2(\Omega\times Q)$, if
\begin{equation}\label{eq:136}
\lim_{\varepsilon\to 0}\int_{\Omega\times Q} u\e(\omega,x)  \varphi(\tau_{\frac{x}{\varepsilon}}\omega,x) \; dP(\omega)dx = \int_{\Omega\times Q}u(\omega,x) \varphi(\omega,x) \; dP(\omega)dx,
\end{equation}
for all \textit{admissible} test functions $\varphi \in L^2(\Omega\times Q)$, see Remark~\ref{remark:242} for details.
\item More recently, Zhikov and Pyatnitskii introduced in \cite{Zhikov2006} a \textit{quenched} variant: We say that a sequence $(u\e)\subset L^2(Q)$ \textit{quenched stochastically two-scale} converges to $u\in L^2(\Omega \times Q)$ w.r.t.~to a fixed parameter field $\omega_0 \in \Omega$, if
\begin{equation*}
\lim_{\varepsilon\to 0}\int_{Q}u\e(x) \varphi(\tau_{\frac{x}{\varepsilon}}\omega_0,x) \;dx= \int_{\Omega\times Q}u(\omega,x) \varphi(\omega,x) \; dP(\omega)dx, 
\end{equation*}
for all \textit{admissible} test functions $\varphi \in L^2(\Omega\times Q)$. Note that the two-scale limit $u$ a priori depends on $\omega_0$. In fact, in \cite{Zhikov2000} (see also \cite{heida2011extension}) quenched two-scale convergence has been introduced in a very general setting that includes the case of integration against random, rapidly oscillating measures, which naturally emerge when describing coefficients defined relative to random geometries. In this work, we restrict our considerations to the simplest case where the random measure is the Lebesgue measure.
\end{itemize}
Similarly to the periodic case, stochastic two-scale convergence in the mean can be rephrased with help of a transformation operator, see \cite{neukamm2017stochastic,heida2019stochastic,varga2019stochastic}, where the \textit{stochastic unfolding operator} $\unf :L^2(\Omega\times Q)\to L^2(\Omega\times Q)$,
\begin{equation}\label{eq:173:3}
\unf u(\omega,x)= u(\tau_{-\frac{x}{\varepsilon}}\omega,x),
\end{equation}
has been introduced. As in the periodic case, it is a linear isometry and it turns out that for a bounded sequence $(u_\varepsilon)\subseteq L^2(\Omega\times Q)$, stochastic two-scale convergence in the mean is equivalent to weak convergence of the unfolded sequence $\unf u_\varepsilon$. As we demonstrate below in Section~\ref{sec:nonerg}, the stochastic unfolding method leads to a very economic and streamlined analysis of convex homogenization problems. Moreover, it allows us to derive two-scale functionals of the form $\mathcal E(u,\chi)=\int_\Omega\int_QV(\omega,\nabla u(x)+\chi(\omega,x))\,dx\,dP$ as a $\Gamma$-limit of $\mathcal E_\varepsilon$, see Theorem~\ref{thm1} for details. In contrast to the periodic case, where the unfolding operator is an isometry from $L^2(\R^d)$ to $L^2(\mathcal{Y} \times \R^d)$ (with $\mathcal{Y}$ denoting the unit torus), in the random case it is not possible to interpret \eqref{eq:173:3} as a continuous operator from $L^2(Q)$ to $L^2(\Omega \times Q)$. Therefore, quenched two-scale convergence cannot be characterized via stochastic unfolding directly.
\smallskip

In the present paper we compare the different notions of stochastic two-scale convergence. Although the mean and quenched notion of two-scale convergence look quite similar, it is non-trivial to relate both. As a main result, we introduce stochastic two-scale Young measures as a tool to compare quenched and mean limits, see Theorem~\ref{thm:Balder-Thm-two-scale}. The construction invokes a metric characterization of quenched stochastic two-scale convergence, which is a tool of independent interest, see Lemma~\ref{L:metric-char}. As an application we demonstrate how to lift a mean two-scale homogenization result to a quenched statement, see Section~\ref{Section:4:3}. Moreover, we present two examples that can only be conveniently treated with the  mean notion of two-scale convergence.  In the first example, see Section \ref{sec:nonerg},  the assumption of ergodicity is dropped (as it is natural in the context of periodic representative volume approximation schemes). In the second example we consider a model that invokes a mean field interaction in form of a \textit{variance-type} regularization of a convex integral functional with degenerate growth, see Section \ref{sec:1012}.
\medskip

\textbf{Structure of the paper.} In the following section we present the standard setting for stochastic homogenization. In Section \ref{S_Stoch} we provide the main properties of the stochastic unfolding method, present the most important facts about quenched two-scale convergence and present our results about Young measures. In Section \ref{Section_Applications} we present examples of stochastic homogenization and applications of the methods developed in this paper.
\medskip

\section{Standard model of random coefficients} In the following we briefly recall the standard setting for stochastic homogenization. Throughout the entire paper we assume the following:
\begin{assumption} \label{Assumption_2_1}
Let $\brac{\Omega,\mathcal{F},P}$ be a complete and separable probability space. Let $\tau=\cb{\tau_x}_{x\in \R^{d}}$ denote a group of invertible measurable mappings $\tau_x:\Omega\to \Omega$ such that:
\begin{enumerate}[label=(\roman*)]
\item (Group property). $\tau_0=Id$ and $\tau_{x+y}=\tau_x\circ \tau_y$ for all $x,y\in \R^{d}$.
\item (Measure preservation). $P(\tau_{-x} E)=P(E)$ for all $E\in \mathcal{F}$ and $x\in \R^{d}$.
\item (Measurability). $(\omega,x)\mapsto \tau_{x}\omega$ is $\brac{\mathcal{F}\otimes \mathcal{L}(\R^d),\mathcal{F}}$-measurable, where  $\mathcal{L}(\R^d)$ denotes the Lebesgue $\sigma$-algebra.

\end{enumerate}
\end{assumption}
We write $\langle\cdot\rangle$ to denote the expectation $\int_\Omega\cdot\, dP$. By the separability assumption on the measure space it follows that $L^p(\Omega)$ is separable for $p\geq 1$. The proof of the following lemma is a direct consequence of Assumption \ref{Assumption_2_1}, thus we omit it. 

\begin{lemma}[Stationary extension]\label{L:stat}
  Let $\varphi:\Omega\to\R$ be $\mathcal F$-measurable. Let $Q\subset \R^d$ be open and denote by $\mathcal{L}(Q)$ the corresponding Lebesgue $\sigma$-algebra. Then $S\varphi:\Omega\times Q\to\R$, $S\varphi(\omega,x):=\varphi(\tau_x\omega)$ defines an $\mathcal F\otimes\mathcal L(Q)$-measurable function -- called the stationary extension of $\varphi$. Moreover, if $Q$ is bounded, for all $1\leq p<\infty$ the map $S:L^p(\Omega)\to L^p(\Omega\times Q)$ is a linear injection satisfying
  \begin{equation*}
    \|S\varphi\|_{L^p(\Omega\times Q)}=|Q|^\frac{1}{p}\|\varphi\|_{L^p(\Omega)}.
  \end{equation*}
\end{lemma}
We say $(\Omega,\mathcal F,P,\tau)$ is \textit{ergodic} ($\ex{\cdot}$ is ergodic), if
\begin{align*}
  \text{ every shift invariant }A\in \mathcal{F} \text{ (i.e.~}\tau_x A=A \text{ for all }x\in \re d)\text{ satisfies } P(A)\in \cb{0,1}. 
\end{align*}
In this case the celebrated Birkhoff's ergodic theorem applies, which we recall in the following form:
\begin{thm}[{Birkhoff's ergodic Theorem \cite[Theorem 10.2.II]{Daley1988}}]
\label{thm:ergodic-thm} Let $\ex{\cdot}$ be ergodic and  $\varphi:\Omega\to\R$ be integrable. Then for $P$-a.a.~$\omega\in\Omega$ it holds: $ S\varphi(\omega,\cdot)$ is locally integrable and for all open, bounded sets $Q\subset\R^d$ we have
\begin{equation}
\lim_{\eps\rightarrow0}\int_{Q}S\varphi(\omega,\tfrac{x}{\varepsilon})\,dx=|Q|\langle\varphi\rangle\,.\label{eq:ergodic-thm}
\end{equation}
Furthermore, if $\varphi\in L^p(\Omega)$ with $1\leq p\leq\infty$, then for $P$-a.a.~$\omega\in\Omega$ it holds: $S\varphi(\omega,\cdot)\in L_{loc}^{p}(\R^d)$, and provided $p<\infty$ it holds $S\varphi(\omega,\frac{\cdot}{\varepsilon})\weakto \langle\varphi\rangle$ weakly in $L_{loc}^{p}(\R^d)$ as $\eps\rightarrow0$.
\end{thm}
\textbf{Stochastic gradient.} For $p \in (1,\infty)$ consider the group of isometric operators $\cb{U_x:x\in \re{d}}$ on $L^p(\Omega)$ defined by $U_x \varphi(\omega)=\varphi(\tau_{x}\omega)$. This group is strongly continuous (see \cite[Section 7.1]{jikov2012homogenization}). For $i=1,...,d$, we consider the 1-parameter group of operators $\cb{U_{h e_i}:h\in \re{}}$ and its infinitesimal generator $D_i:\mathcal{D}_i\subset L^p(\Omega)\rightarrow L^p(\Omega)$
\begin{equation*}
D_i \varphi=\lim_{h\rightarrow 0} \frac{U_{he_i}\varphi-\varphi}{h},
\end{equation*}
which we refer to as \textit{stochastic derivative}. $D_i$ is a linear and closed operator and its domain $\mathcal{D}_i$ is dense in $L^p(\Omega)$. We set $W^{1,p}(\Omega)=\cap_{i=1}^{d}\mathcal{D}_i$ and define for $\varphi\in W^{1,p}(\Omega)$ the stochastic gradient as $D\varphi=(D_1 \varphi,...,D_d \varphi)$. In this way, we obtain  a linear, closed and densely defined operator $D:W^{1,p}(\Omega)\rightarrow L^p(\Omega)^d$, and we denote by
\begin{equation}\label{eq:262}
L^p_{\pot}(\Omega):=\overline{\mathcal R(D)}\subset L^p(\Omega)^d
\end{equation}
the closure of the range of $D$ in $L^p(\Omega)^d$. We denote the adjoint of $D$ by $D^*:\mathcal{D}^*\subset{L^q(\Omega)^d}\rightarrow L^q(\Omega)$ where here and below  $q:= \frac{p}{p-1}$ denotes the dual exponent. It is a linear, closed and densely defined operator ($\mathcal{D}^*$ is the domain of $D^*$). We define the subspace of shift invariant functions in $L^p(\Omega)$ by 
\begin{equation*}
L^p_{{\inv}}(\Omega)=\cb{\varphi\in L^p(\Omega):U_x \varphi=\varphi \quad \text{for all }x \in \re{d}},
\end{equation*}
and denote by $ P_{\inv}:L^p(\Omega) \to L^p_{\inv}(\Omega)$ the conditional expectation with respect to the $\sigma$-algebra of shift invariant sets $\cb{ A \in \mathcal{F} : \tau_x A = A \text{ for all } x\in \re d}$. $P_{\inv}$ a contractive projection and for $p=2$ it coincides with the orthogonal projection onto $L^2_{\inv}(\Omega)$. The following well-known equivalence holds:
\begin{equation*}
\text{$\ex{\cdot}$ is ergodic $\Leftrightarrow$ $L^p_{\inv}(\Omega)\simeq \re{}$ $\Leftrightarrow$ $\pinv f=\ex{f}$.} 
\end{equation*}

\textbf{Random fields.}  We introduce function spaces for functions defined on $\Omega\times Q$ as follows: For closed subspaces $X\subset L^p(\Omega)$ and $Y\subset L^p(Q)$, we denote by $X\otimes Y$ the closure of $$X\overset{a}{\otimes}Y:=\cb{\sum_{i=1}^{n}\varphi_i \eta_i:  \varphi_i \in X, \eta_i\in Y, n\in \mathbb{N}}$$ in $L^p(\Omega\times Q)$. Note that in the case $X=L^p(\Omega)$ and $Y=L^p(Q)$, we have $X\otimes Y = L^p(\Omega\times Q)$. Up to isometric isomorphisms, we may identify $L^p(\Omega\times Q)$ with the Bochner spaces $L^p(\Omega;L^p(Q))$ and $L^p(Q;L^p(\Omega))$. Slightly abusing the notation, for closed subspaces $X\subset L^p(\Omega)$ and $Y\subset W^{1,p}(Q)$, we denote by $X \otimes Y$ the closure of
$$X\overset{a}{\otimes}Y:=\cb{\sum_{i=1}^{n}\varphi_i \eta_i:  \varphi_i \in X, \eta_i\in Y, n\in \mathbb{N}}$$
in $L^p(\Omega;W^{1,p}(Q))$. In this regard, we may identify $u\in L^p(\Omega)\otimes W^{1,p}(Q)$ with the pair $(u,\nabla u)\in L^p(\Omega\times Q)^{1+d}$. We mostly focus on the space $L^p(\Omega\times Q)$ and the above notation is convenient for keeping track of its various subspaces.

\section{Stochastic two-scale convergence, unfolding and Young measures}\label{S_Stoch}

In the following we first discuss two notions of stochastic two-scale convergence and their connection through Young measures. In particular, Section \ref{S_Stoch_1} is devoted to the introduction of the stochastic unfolding operator and its most important properties. In Section \ref{section:312} we discuss quenched two-scale convergence and its properties. Section \ref{Section_4} presents the results about Young measures.

\subsection{Stochastic unfolding and two-scale convergence in the mean}\label{S_Stoch_1}
In the following we briefly introduce the stochastic unfolding operator and provide its main properties, for the proofs and detailed studies we refer to \cite{neukamm2017stochastic,heida2019stochastic,varga2019stochastic,NeukammVargaWaurick}.
\begin{lemma}[{\cite[Lemma 3.1]{heida2019stochastic}}]\label{L:unf}
  Let $\varepsilon>0$, $1<p<\infty$, $q=\frac{p}{p-1}$, and $Q\subset\R^d$ be open. There exists a unique linear isometric isomorphism
  \begin{equation*}
    \unf: \ltp\rightarrow \ltp
  \end{equation*}
  such that 
  \begin{equation*}
    \forall u\in L^p(\Omega)\overset{a}{\otimes} L^p(Q)\,:\qquad (\unf u)(\omega,x)=u(\tau_{-\frac{x}{\varepsilon}}\omega,x)\qquad \text{a.e. in }\Omega\times Q.
  \end{equation*}
  Moreover, its adjoint is the unique linear isometric isomorphism $\unf^{*}:\ltq\to\ltq$ that satisfies $(\unf^{*}u)(\omega,x)=u(\tau_{\frac{x}{\varepsilon}}\omega,x)$ a.e.~in $\Omega\times Q$ for all $u\in L^q(\Omega)\overset{a}{\otimes}L^q(Q)$, $q:= \frac{p}{p-1}$.
\end{lemma}
\begin{defn}[Unfolding and two-scale convergence in the mean]\label{def46}
  The operator $\unf:\ltp\to\ltp$ in Lemma~\ref{L:unf} is called the stochastic unfolding operator. We say that a sequence $(u\e) \subset \ltp$ weakly (strongly) two-scale converges in the mean in $\ltp$ to $u\in \ltp$ if (as $\varepsilon\to 0$)
  \begin{equation*}
    \unf u\e \rightarrow u \quad \text{ weakly (strongly) in }\ltp.
  \end{equation*} 
  In this case we write $u\e\wt u$ ($u_{\varepsilon} \st u$) in $L^p(\Omega\times Q)$.
\end{defn}

\begin{remark}[Equivalence to stochastic two-scale convergence in the mean]\label{remark:242}
Stochastic two-scale convergence in the mean was introduced in \cite{bourgeat1994stochastic}. In particular, it is said that a sequence of random fields $u\e \in L^p(\Omega\times Q)$ stochastically two-scale converges in the mean if
\begin{equation}\label{eq:1}
\lim_{\varepsilon\rightarrow 0}\ex{\int_{Q}u_{\varepsilon}(\omega,x)\varphi(\tau_{\frac{x}{\varepsilon}}\omega,x)dx}=\ex{ \int_{Q}u(\omega,x)\varphi(\omega,x)dx},
\end{equation}
for any $\varphi\in L^q(\Omega\times Q)$,  $q= \frac{p}{p-1}$, that is admissible, i.e., in the sense that the transformation $(\omega,x)\mapsto \varphi(\tau_{\frac{x}{\varepsilon}}\omega,x)$ is well-defined. For a bounded sequence $u\e \in L^p(\Omega\times Q)$, \eqref{eq:1} is equivalent to $\unf u\e \rightharpoonup u$ weakly in $L^p(\Omega\times Q)$, i.e., to weak stochastic two-scale convergence in the mean. Indeed, with help of $\unf$ (and its adjoint) we might rephrase the integral on the left-hand side in \eqref{eq:1} as
\begin{equation}\label{eq:1234}
\ex{\int_{Q}u_{\varepsilon}(\unf^{*}\varphi)\, dx}=\ex{\int_{Q}(\unf u_{\varepsilon})\varphi dx},
\end{equation}
which proves the equivalence.
\end{remark}
We summarize some of the main properties:
\begin{prop}[Main properties]\label{prop:292}
Let $p\in (1,\infty)$, $q=\frac{p}{p-1}$ and $Q\subset \R^d$ be open.
\begin{enumerate}[label = (\roman*)]
\item (Compactness, \cite[Lemma 3.4]{heida2019stochastic}.) If $\limsup_{\varepsilon\rightarrow 0}\ns{u\e}<\infty$, then there exists a subsequence $\varepsilon'$ and $u\in \ltp$ such that $u_{\varepsilon'} \wt u$ in $\ltp$.

\item (Limits of gradients, \cite[Proposition 3.7]{heida2019stochastic}) Let $(u\e)$ be a bounded sequence in $L^p(\Omega)\otimes \sob(Q)$. Then, there exist $u\in L^p_{{\inv}}(\Omega)\otimes \sob(Q)$ and $\chi\in L^p_{\pot}(\Omega)\otimes L^p(Q)$ such that (up to a subsequence)
\begin{equation}\label{equation1}
u\e \wt u  \text{ in }\ltp, \quad \nabla u\e \wt \nabla u +\chi  \text{ in }\ltp^d.
\end{equation}
If, additionally, $\ex{\cdot}$ is ergodic, then $u=P_{\mathsf{{\inv}}} u=\ex{u} \in \sob(Q)$ and $\ex{u_{\varepsilon}}\weakto u$ weakly in $\sob(Q)$.

\item (Recovery sequences, \cite[Lemma 4.3]{heida2019stochastic}) Let $u\in L^p_{{\inv}}(\Omega)\otimes \sob(Q)$ and $\chi\in L^p_{\pot}(\Omega)\otimes L^p(Q)$. There exists $u\e \in L^p(\Omega)\otimes W^{1,p}(Q)$ such that
\begin{equation*}
u\e \overset{2}{\to} u, \quad \nabla u\e \overset{2}{\to}\nabla u + \chi \quad \text{in }L^p(\Omega\times Q).
\end{equation*}
If additionally $u \in L^p_{\mathrm{inv}}(\Omega)\otimes W^{1,p}_0(Q)$, we have $u\e \in L^p(\Omega)\otimes W^{1,p}_0(Q)$.
\end{enumerate}
\end{prop}
\def\per{{\sf per}}
\subsection{Quenched two-scale convergence}\label{section:312}

In this section, we recall the concept of quenched stochastic two-scale convergence (see~\cite{Zhikov2006,heida2011extension}). The notion of quenched stochastic two-scale convergence is based on the individual ergodic theorem, see Theorem~\ref{thm:ergodic-thm}. We thus assume throughout this section that 
\begin{equation*}
\text{$\ex{\cdot}$ is ergodic.}
\end{equation*}
Moreover, throughout this section we fix exponents $p\in(1,\infty)$, $q:=\frac{p}{p-1}$, and an open and bounded domain $Q\subset\R^d$. We denote by $(\sB^p, \|\cdot\|_{\sB^p})$ the Banach space $L^p(\Omega\times Q)$ and the associated norm, and we write $(\sB^p)^*$ for the dual space. For the definition of quenched two-scale convergence we need to specify a suitable space of test-functions in $\sB^q$ that is countably generated. To that end we fix sets $\sD_\Omega$ and $\sD_Q$ such that
\begin{itemize}
\item $\sD_\Omega$ is a countable set of bounded, measurable functions on $(\Omega,\mathcal F)$ that contains the identity $\mathbf 1_{\Omega}\equiv 1$ and is dense in $L^1(\Omega)$ (and thus in $L^r(\Omega)$ for any $1\leq r<\infty$).
\item $\sD_Q\subset C(\overline Q)$ is a countable set that contains the identity $\mathbf 1_Q\equiv 1$ and is dense in $L^1(Q)$ (and thus in $L^r(Q)$ for any $1\leq r<\infty$).
\end{itemize}
We denote by 
$$\sA:=\{\varphi(\omega,x)=\varphi_\Omega(\omega)\varphi_Q(x)\,:\,\varphi_\Omega\in\sD_\Omega,\varphi_Q\in\sD_Q\}$$ 

the set of simple tensor products (a countable set), and by $\sD_0$ the $\mathbb Q$-linear span of $\sA$, i.e.~
\begin{equation*}
  \sD_0:=\big\{\,\sum_{j=1}^m\lambda_j\varphi_j\,:\,m\in\N,\,\lambda_1,\ldots,\lambda_m\in\mathbb Q,\,\varphi_1,\ldots,\varphi_m\in\sA\,\big\}.
\end{equation*}
We finally set 
$$\sD:=\mbox{span}\sA=\mbox{span}\sD_0\quad \text{and}\quad \overline{\sD}:=\mbox{span}(\sD_Q)$$
 (the span of $\sD_Q$ seen as a subspace of $\sD$), and note that $\sD$ and $\sD_0$ are dense subsets of $\sB^q$, while the closure of $\overline{\sD}$ in $\sB^q$ is isometrically isomorphic to $L^q(Q)$. Let us anticipate that $\sD$ serves as our space of test-functions for stochastic two-scale convergence. As opposed to two-scale convergence in the mean, ``quenched'' stochastic two-scale convergence is defined relative to a fixed ``admissible'' realization $\omega_0\in\Omega$. Throughout this section we denote by 
$$\Omega_0 \quad \text{the set of admissible realizations;}$$ it is a set of full measure determined by the following lemma:
\begin{lemma}\label{L:admis}
  There exists a measurable set $\Omega_0\subset\Omega$ with $P(\Omega_0)=1$ s.t. for all $\varphi,\varphi'\in\sA $, all $\omega_0\in\Omega_0$, and $r\in\{p,q\}$ we have with $(\unf^*\varphi)(\omega,x):=\varphi(\tau_{\frac{x}{\eps}}\omega,x)$,
  \begin{align*}
    \limsup\limits_{\eps\to 0}\|(\unf^*\varphi)(\omega_0,\cdot)\|_{L^r(Q)}&\leq \|\varphi\|_{\sB^r}\\ \text{and}\qquad \lim\limits_{\eps\to 0}\int_Q\unf^*(\varphi\varphi')(\omega_0,x)dx&=\ex{\int_Q(\varphi\varphi')(\omega_0,x)\,dx}.
  \end{align*}
\end{lemma}
\begin{proof}
  This is a simple consequence of Theorem~\ref{thm:ergodic-thm} and the fact that $\sA$ is countable.
\end{proof}
For the rest of the section $\Omega_0$ is fixed according to Lemma~\ref{L:admis}.
\medskip

The idea of quenched stochastic two-scale convergence is similar to periodic two-scale convergence: We associate with a bounded sequence $(u\e)\subset L^p(Q)$ and $\omega_0\in\Omega_0$, a sequence of linear functionals $(U\e)$ defined on $\sD$. We can pass (up to a subsequence) to a pointwise limit $U$, which is again a linear functional on $\sD$ and which (thanks to Lemma~\ref{L:admis}) can be uniquely extended to a bounded linear functional on $\sB^q$. We then define the \textit{weak quenched $\omega_0$-two-scale limit} of $(u\e)$ as the Riesz-representation $u\in \sB^p$ of $U\in(\sB^q)^*$.
\begin{defn}[quenched two-scale limit, cf.~\cite{Zhikov2006,Heida2017b}]
\label{def:two-scale-conv}
Let $(u\e)$ be a sequence in $L^{p}(Q)$, and let $\omega_0\in\Omega_0$ be fixed.  We say that $u\e$ converges (weakly, quenched) $\omega_0$-two-scale to $u\in \sB^{p}$, and write
$u\e\tsq{\omega_0}u$, if the sequence $u\e$ is bounded in $L^p(Q)$, and for all $\varphi\in \sD$ we have
\begin{equation}
  \lim_{\eps\to0}\int_{Q}u\e(x)(\unf^*\varphi)(\omega_0,x)\,dx=\int_\Omega\int_{Q}u(x,\omega)\varphi(\omega,x)\,dx\,dP(\omega).\label{eq:def-quenched-two-scale}
\end{equation}
\end{defn}
\begin{lemma}[Compactness]\label{lem:two-scale-limit}
Let $(u\e)$ be a bounded sequence in $L^p(Q)$ and $\omega_0\in \Omega_0$. Then there exists a subsequence (still denoted by $\eps$) and $u\in \sB^p$ such that $u_{\eps}\tsq{\omega_0}u$ and
    \begin{equation}
      \| u\|_{\sB^{p}}\leq \liminf_{\eps\to0}\|u_{\eps}\|_{L^{p}(Q)},\,\label{eq:two-scale-limit-estimate}
    \end{equation}
    and $u_{\eps}\weakto \ex{u}$ weakly in $L^p(Q)$.
\end{lemma}
{\it (For the proof see Section~\ref{S:4:p1}).}\smallskip

For our purpose it is convenient to have a metric characterization of two-scale convergence.
\begin{lemma}[Metric characterization]\label{L:metric-char}
  \begin{enumerate}[label=(\roman*)]
  \item 
  Let $\{\varphi_j\}_{j\in\N}$ denote an enumeration of the countable set  $\{\frac{\varphi}{\|\varphi\|_{\sB^q}}\,:\,\varphi\in \sD_0\}$. The vector space $\mbox{\rm Lin}(\sD):=\{U:\sD\to\R\text{ linear}\,\}$ endowed with the metric $$d(U,V;\mbox{\rm Lin}(\sD)):=\sum_{j\in\N}2^{-j}\frac{|U(\varphi_j)-V(\varphi_j)|}{|U(\varphi_j)-V(\varphi_j)|+1}$$  is complete and separable.
\item   Let $\omega_0\in\Omega_0$. Consider the maps
  \begin{align*}
    &J_\eps^{\omega_0}: L^p(Q)\to \mbox{\rm Lin}(\sD),\qquad (J_\eps^{\omega_0} u)(\varphi):=\int_Qu(x)(\unf^*\varphi)(\omega_0,x)\,dx,\\
    &J_0:\sB^p\to \mbox{\rm Lin}(\sD),\qquad (J_0u)(\varphi):=\ex{\int_Qu\varphi}.
  \end{align*}
  Then for any bounded sequence $u_\eps$ in $L^p(Q)$ and any $u\in \sB^p$ we have $u_\eps\tsq{\omega_0}u$ if and only if $J_\eps^{\omega_0} u_\eps \to J_0u$ in $\mbox{\rm Lin}(\sD)$.
\end{enumerate}
\end{lemma}
{\it (For the proof see Section~\ref{S:4:p1}).}\smallskip
\begin{remark}
  Convergence in the metric space $(\mbox{\rm Lin}(\sD),d(\cdot,\cdot,\mbox{\rm Lin}(\sD))$ is equivalent to pointwise convergence. $(\sB^q)^*$ is naturally embedded into the metric space by means of the restriction $J:(\sB^q)^*\to\mbox{\rm Lin}(\sD)$, $JU=U\vert_{\sD}$. In particular, we deduce that for a bounded sequences $(U_k)$ in $(\sB^q)^*$ we have $U_k\stackrel{*}{\weakto} U$ if and only if $JU_k\to JU$ in the metric space. Likewise, $\sB^p$  (resp. $L^p(Q)$) can be embedded into the metric space $\mbox{\rm Lin}(\sD)$ via $J_0$ (resp. $J_\eps^{\omega_0}$ with $\eps>0$ and $\omega_0\in\Omega_0$ arbitrary but fixed), and for a bounded sequence $(u_k)$ in $\sB^p$ (resp. $L^p(Q)$) weak convergence in $\sB^p$ (resp. $L^p(Q)$) is equivalent to convergence of $(J_0u_k)$ (resp. $(J^{\omega_0}_\eps u_k)$) in the metric space.
\end{remark}
\begin{lemma}[Strong convergence implies quenched two-scale convergence]\label{L:strong}
Let $(u\e)$ be a strongly convergent sequence in $L^p(Q)$ with limit $u\in L^p(Q)$. Then for all $\omega_0\in\Omega_0$ we have $u\e\tsq{\omega_0}u$.  
\end{lemma}
{\it (For the proof see Section~\ref{S:4:p1}).}\smallskip

\begin{defn}[set of quenched two-scale cluster points]
  For a bounded sequence $(u\e)$ in $L^p(Q)$ and $\omega_0\in\Omega_0$ we denote by $\CL(\omega_0,(u\e))$ the set of all $\omega_0$-two-scale cluster points, i.e. the set of $u\in\sB^p$ with  $J_0u\in \bigcap_{k=1}^\infty\overline{\big\{J^{\omega_0}_{\eps} u_\eps\,:\,\eps<\frac{1}{k}\big\}}$  where the closure is taken in the metric space $\big(\mbox{\rm Lin}(\sD),d(\cdot,\cdot;\mbox{\rm Lin}(\sD)))$.
\end{defn}
We conclude this section with two elementary results on quenched stochastic two-scale convergence:
\begin{lemma}[Approximation of two-scale limits]
\label{lem:Every-v-is-a-fB-limit}Let $u\in\sB^p$.
Then for all $\omega_0\in\Omega_0$, there exists a sequence
$u\e\in L^{p}(Q)$ such that $u\e\stackrel{2}{\weakto}_{\omega_0} u$ as $\eps\to0$.
\end{lemma}
{\it (For the proof see Section~\ref{S:4:p1}).}\smallskip

Similar to the slightly different setting in \cite{Heida2017b}
one can prove the following result:
\begin{lemma}[Two-scale limits of gradients]
\label{lem:sto-conver-grad}
Let $(u_{\eps})$ be a sequence in $W^{1,p}(Q)$ and $\omega_0\in\Omega_0$. Then there exist a subsequence (not relabeled) and functions $u\in W^{1,p}(Q)$ and $\chi\in L^p_{\mathrm{pot}}(\Omega)\otimes L^{p}(Q)$ such that $u\e\weakto u$ weakly in $W^{1,p}(Q)$ and
\[
u_{\eps}\tsq{\omega_0}u\quad\mbox{and }\quad\nabla u_{\eps}\tsq{\omega_0}\nabla u+\chi\qquad\mbox{as }\eps\to0\,.
\]
\end{lemma}
\subsubsection{Proofs}\label{S:4:p1}
\begin{proof}[Proof of Lemma~\ref{lem:two-scale-limit}]
  Set $C_0:=\limsup\limits_{\eps\to 0}\|u\e\|_{L^p(Q)}$ and note that $C_0<\infty$. By passing to a subsequence (not relabeled) we may assume that $C_0=\liminf\limits_{\eps\to 0}\|u\e\|_{L^p(Q)}$.
Fix $\omega_0\in\Omega_0$. Define linear functionals $U_\eps\in\mbox{\rm Lin}(\sD)$ via
    \begin{equation*}
      U_\eps(\varphi):=\int_Qu\e(x)(\unf^*\varphi)(\omega_0,x)\,dx.
    \end{equation*}
    Note that for all $\varphi\in\sA$, $(U\e(\varphi))$ is a bounded sequence in $\R$. Indeed, by H\"older's inequality and Lemma~\ref{L:admis},
    \begin{equation}\label{eq:x1}
      \limsup\limits_{\eps\to0}|U\e(\varphi)|\leq      \limsup\limits_{\eps\to0}\|u\e\|_{L^p(Q)}\|\unf^*\varphi(\omega_0,\cdot)\|_{L^q(Q)}\leq C_0\|\varphi\|_{\sB^q}.
    \end{equation}
    Since $\sA$ is countable we can pass to a subsequence (not relabeled) such that $U\e(\varphi)$ converges for all $\varphi\in\sA$. By linearity and since $\sD=\mbox{span}(\sA)$, we conclude that $U\e(\varphi)$ converges for all $\varphi\in\sD$, and $U(\varphi):=\lim\limits_{\eps\to0}U\e(\varphi)$ defines a linear functional on $\sD$. In view of \eqref{eq:x1} we have $|U(\varphi)|\leq C_0\|\varphi\|_{\sB^q}$, and thus $U$ admits a unique extension to a linear functional in $(\sB^q)^*$. Let $u\in\sB^p$ denote its Riesz-representation. Then $u\e\tsq{\omega_0} u$, and
    \begin{equation*}
      \|u\|_{\sB^p}=\|U\|_{(\sB^q)^*}\leq C_0=\liminf\limits_{\eps\to 0}\|u\e\|_{L^p(Q)}.
    \end{equation*}
    Since $\mathbf 1_{\Omega}\in\sD_\Omega$ we conclude that for all $\varphi_Q\in\sD_Q$ we have
    \begin{equation*}
      \int_Qu\e(x)\varphi_Q(x)\,dx=U\e(\mathbf 1_{\Omega}\varphi_Q)\to U(\mathbf 1_{\Omega}\varphi_Q)=\ex{\int_Q u(\omega,x)\varphi_Q(x)\,dx}=\int_Q \ex{u(x)}\varphi_Q(x)\,dx.
    \end{equation*}
    Since $(u\e)$ is bounded in $L^p(Q)$, and $\sD_Q\subset L^p(Q)$ is dense, we conclude that $u\e\weakto\ex{u}$ weakly in $L^p(Q)$.
\end{proof}

\begin{proof}[Proof of Lemma~\ref{L:metric-char}]
We use the following notation in this proof $\sA_1:=\{\frac{\varphi}{\|\varphi\|_{\sB^q}}\,:\,\varphi\in \sD_0\}$.

(i) Argument for completeness: If $(U_j)$ is a Cauchy sequence in $\mbox{\rm Lin}(\sD)$, then for all $\varphi\in\sA_1$, 
    $(U_j(\varphi))$ is a Cauchy sequence in $\R$. By linearity of the $U_j$'s this implies that $(U_j(\varphi))$ is Cauchy in $\R$ for all $\varphi\in\sD$. Hence, $U_j\to U$ pointwise in $\sD$ and it is easy to check that $U$ is linear. Furthermore, $U_j\to U$ pointwise in $\sA_1$ implies $U_j\to U$ in the metric space.
    
Argument for separability: Consider the (injective) map $J:(\sB^q)^*\to\mbox{\rm Lin}(\sD)$ where $J(U)$ denotes the restriction of $U$ to $\sD$. The map $J$ is continuous, since for all $U,V\in(\sB^q)^*$ and $\varphi\in\sA_1$ we have $|(JU)(\varphi)-(JV)(\varphi)|\leq \|U-V\|_{(\sB^q)^*}\|\varphi\|_{\sB^q}=\|U-V\|_{(\sB^q)^*}$ (recall that the test functions in $\sA_1$ are normalized). Since $(\sB^q)^*$ is separable (as a consequence of the assumption that $\mathcal F$ is countably generated), it suffices to show that the range $\mathcal R(J)$ of $J$ is dense in $\mbox{\rm Lin}(\sD)$. To that end let $U\in\mbox{\rm Lin}(\sD)$. For $k\in\N$ we denote by $U_k\in(\sB^q)^*$ the unique linear functional that is equal to $U$ on the the finite dimensional (and thus closed) subspace $\mbox{span}\{\varphi_1,\ldots,\varphi_k\}\subset \sB^q$ (where $\{\varphi_j\}$ denotes the enumeration of $\sA_1$), and zero on the orthogonal complement in $\sB^q$. Then a direct calculation shows that $d(U,J(U_k);\mbox{\rm Lin}(\sD))\leq\sum_{j>k}2^{-j}=2^{-k}$. Since $k\in\N$ is arbitrary, we conclude that $\mathcal R(J)\subset\mbox{\rm Lin}(\sD)$ is dense.

(ii) Let $u\e$ denote a bounded sequence in $L^p(Q)$ and $u\in \sB^p$. Then by definition, $u\e\tsq{\omega_0}u$ is equivalent to $J^{\omega_0}\e u\e\to J_0u$ pointwise in $\sD$, and the latter is equivalent to convergence in the metric space $\mbox{\rm Lin}(\sD)$.
\end{proof}

\begin{proof}[Proof of Lemma~\ref{L:strong}]
This follows from H{\"o}lder's inequality and Lemma~\ref{L:admis}, which imply for all $\varphi\in\sA$ the estimate
\begin{multline*}
  \limsup\limits_{\eps\to 0}\int_Q|(u\e(x)-u(x))\unf^*\varphi(\omega_0,x)|\,dx\\ 
	\leq    \limsup\limits_{\eps\to 0}\Big(\|u\e-u\|_{L^p(Q)}\left(\int_Q|\unf^*\varphi(\omega_0,x)|^q\,dx\right)^\frac1q\Big)=0.
\end{multline*}
\end{proof}

\begin{proof}[Proof of Lemma~\ref{lem:Every-v-is-a-fB-limit}]
  Since $\sD$ (defined as in Lemma~\ref{L:metric-char}) is dense in $\sB^p$, for any $\delta>0$ there exists $v_\delta\in\sD_0$ with $\|u-v_\delta\|_{\sB^p}\leq \delta$. Define $v_{\delta,\eps}(x):=\unf^*v_\delta(\omega_0,x)$. Let $\varphi\in\sD$. Since $v_\delta$ and $\varphi$ (resp. $v_\delta\varphi$) are by definition linear combinations of functions (resp. products of functions) in $\sA$, we deduce from Lemma~\ref{L:admis} that $(v_{\delta,\eps})_{\eps}$ is bounded in $L^p(Q)$, and that 
  \begin{equation*}
    \int_Qv_{\delta,\eps}\unf^*\varphi(\omega_0,x)=\int_Q\unf^*(v_\delta\varphi)(\omega_0,x)\to \ex{\int_Qv_\delta\varphi}.
  \end{equation*}
  By appealing to the metric characterization, we can rephrase the last convergence statement as $d(J^{\omega_0}\e v_{\delta,\eps},J_0v_\delta;\mbox{\rm Lin}(\sD))\to 0$. By the triangle inequality we have
  \begin{eqnarray*}
    d(J^{\omega_0}\e v_{\delta,\eps},J_0u;\mbox{\rm Lin}(\sD))&\leq& d(J^{\omega_0}\e v_{\delta,\eps},J_0v_\delta;\mbox{\rm Lin}(\sD))+d(J_0v_\delta,J_0u;\mbox{\rm Lin}(\sD)).
  \end{eqnarray*}
  The second term is bounded by $\|v_\delta-u\|_{\sB^p}\leq \delta$, while the first term vanishes for $\eps\downarrow 0$. Hence, there exists a diagonal sequence $u\e:=v_{\delta(\eps),\eps}$ (bounded in $L^p(Q)$) such that there holds $d(J^{\omega_0}\e u_{\eps},J_0u;\mbox{\rm Lin}(\sD))\to 0$. The latter implies $u_{\eps}\tsq{\omega_0}u$ by Lemma~\ref{L:metric-char}.
\end{proof}

\subsection{Young measures generated by two-scale convergence}\label{Section_4}
In this section we establish a relation between quenched two-scale
convergence and two-scale convergence in the mean (in the sense of Definition
\ref{def46}). The relation is established by Young measures: We show that any bounded sequence $u\e$ in $\sB^p$ -- viewed as a functional acting on test-functions of the form $\unf^*\varphi$ -- generates (up to a subsequence) a Young measure on $\sB^p$ that (a) concentrates on the quenched two-scale cluster points of $u\e$, and (b) allows to represent the two-scale limit (in the mean) of $u\e$. In entire Section \ref{Section_4} we assume that
\begin{equation*}
\ex{\cdot} \text{ is ergodic.}
\end{equation*}
Also, throughout this section we fix exponents $p\in(1,\infty)$, $q:=\frac{p}{p-1}$, and an open and bounded domain $Q\subset\R^d$. Furthermore, we frequently use the objects and notations introduced in Section \ref{section:312}.
\begin{defn}
  We say $\boldsymbol{\nu}:=\left\{ \nu_{\omega}\right\} _{\omega\in\Omega}$ is a Young measure on $\sB^p$, if for all $\omega\in\Omega$, $\nu_\omega$ is a Borel probability measure on $\sB^p$  (equipped with the weak topology) and 
  \[
  \omega\mapsto\nu_{\omega}(B)\quad\mbox{is }\mbox{measurable for all }B\in\cB(\sB^p),
  \]
  where $\cB(\sB^p)$ denotes the Borel-$\sigma$-algebra on $\sB^p$  (equipped with the weak topology).
\end{defn}
\begin{thm}
\label{thm:Balder-Thm-two-scale}Let $u_{\eps}$ denote a bounded sequence in $\sB^p$.
Then there exists a subsequence (still denoted by $\eps$) and a Young measure  $\boldsymbol{\nu}$ on $\sB^p$
such that for all $\omega_0\in\Omega_0$,
\[
\nu_{\omega_0}\mbox{ is concentrated on }\CL\left(\omega_0,\big(u_{\eps}(\omega_0,\cdot)\big)\right),
  \]
  and 
  \[
  \liminf_{\eps\to0}\Vert u_{\eps}\Vert_{\sB^p}^{p}\geq \int_{\Omega}\left(\int_{\sB^p}\left\Vert v\right\Vert _{\sB^p}^{p}\,d\nu_{\omega}(v)\right)\,dP(\omega).
  \]
  Moreover, we have 
  \[
  u_{\eps}\ts u\qquad\text{where }u:=\int_{\Omega}\int_{\sB^p}v\, d\nu_{\omega}(v)dP(\omega).
  \]
  Finally, if there exists $v:\Omega\to\sB^p$ measurable and $\nu_{\omega}=\delta_{v(\omega)}$ for $P$-a.a. $\omega\in\Omega$,
  then up to extraction of a further subsequence (still denoted by $\eps$) we have
  \[
  u_{\eps}(\omega)\tsq{\omega}v(\omega)\qquad\text{for $P$-a.a.~$\omega\in\Omega$}.
  \]
\end{thm}
{\it (For the proof see Section~\ref{S:4:p2}).}\smallskip

In the opposite direction we observe that quenched two-scale convergence implies two-scale convergence in the mean in the following sense:
  \begin{lemma}\label{L:fromquenchedtomean}
    Consider a family $\{(u\e^\omega)\}_{\omega\in\Omega}$ of sequences $(u^\omega\e)$ in $L^p(Q)$ and suppose that:
    \begin{enumerate}[label=(\roman*)]
    \item There exists $u\in\sB^p$ s.t.~for $P$-a.a. $\omega\in\Omega$ we have $u\e^\omega\tsq{\omega}u$.
    \item There exists a sequence $(\tilde u\e)$ s.t.~$u\e^\omega(x)=\tilde u\e(\omega,x)$ for a.a.~$(\omega,x)\in\Omega\times Q$.
    \item There exists a bounded sequence $(\chi\e)$ in $L^p(\Omega)$ such that $\|u^\omega\e\|_{L^p(Q)}\leq\chi\e(\omega)$ for a.a.~$\omega\in\Omega$.
    \end{enumerate}
    Then $\tilde u\e\wt u$ weakly two-scale (in the mean).
  \end{lemma}
{\it (For the proof see Section~\ref{S:4:p2}).}\smallskip

To compare homogenization of convex integral functionals w.r.t.~stochastic two-scale convergence in the mean and in the quenched sense, we appeal to the following result:
\begin{lemma}
\label{lem:Balder-Lem-two-scale}
Let $h:\,\Omega\times Q\times\R^{d}\to\R$ be such that for all $\xi\in\R^d$, $h(\cdot,\cdot,\xi)$ is $\mathcal F\otimes\cB(\R^{d})$-measurable and for a.a.~$(\omega,x)\in\Omega\times Q$, $h(\omega,x,\cdot)$ is convex. Let $(u\e)$ denote a bounded sequence in $\sB^p$ that generates a Young measure $\boldsymbol{\nu}$ on $\sB^p$ in the sense of Theorem \ref{thm:Balder-Thm-two-scale}. 
Suppose that $h\e:\Omega\to\R$, $h\e(\omega):=-\int_Q\min\big\{0,h(\tau_{\frac{x}{\eps}}\omega,x,u_{\eps}(\omega,x))\big\}\,dx$ is uniformly integrable. Then 
\begin{multline}
  \liminf_{\eps\to 0}\int_{\Omega}\int_{Q}h(\tau_{\frac{x}{\eps}}\omega,x,u_{\eps}(\omega,x))\,dx\,dP(\omega)\\ 
	\geq\int_{\Omega}\int_{\sB^p}\left(\int_\Omega\int_Qh(\tomega,x,v(\tomega,x))\,dx\,dP(\tomega)\right)\,d\nu_{\omega}(v)\,dP(\omega).\label{eq:liminf-balder-ts-lower-semic}
\end{multline}
\end{lemma}
{\it (For the proof see Section~\ref{S:4:p2}).}
\begin{remark}\label{lem:General-Hom-Convex}
In \cite[Lemma~5.1]{HeidaNesenenko2017monotone} it is shown that $h$ satisfying the assumptions of Lemma~\ref{lem:Balder-Lem-two-scale} satisfies the following property: For $P$-a.a.~$\omega_0\in\Omega_0$ we have: For any sequence $(u\e)$ in $L^p(Q)$ it holds
\begin{equation}\label{ass:lsc}
u\e\tsq{\omega_0}u\qquad\Rightarrow\qquad
\liminf_{\eps\to0}\int_{Q}h(\tau_{\frac{x}{\eps}}\omega_0,x,u\e(x))dx\geq\int_{\Omega}\int_{Q}h(\omega,x,u(\omega,x))\,dx\,dP(\omega).
\end{equation}
\end{remark}

\subsubsection{Proof of Theorem~\ref{thm:Balder-Thm-two-scale} and Lemmas~\ref{lem:Balder-Lem-two-scale} and \ref{L:fromquenchedtomean}}\label{S:4:p2}
We first recall some notions and results of Balder's theory for Young measures \cite{Balder1984}. Throughout this section $\sM$ is assumed to be a separable, complete  metric space with metric $d(\cdot,\cdot;\sM)$. 

\begin{defn}
  \begin{itemize}
  \item We say a function $s:\Omega\to\sM$ is measurable, if it is $\mathcal F-\mathcal B(\sM)$-measurable where $\mathcal B(\sM)$ denotes the Borel-$\sigma$-algebra in $\sM$.
  \item A function $h:\Omega\times\sM\to(-\infty,+\infty]$ is called a \textit{normal} integrand, if $h$ is $\mathcal F\otimes\mathcal B(\sM)$-measurable, and for all $\omega\in\Omega$ the function $h(\omega,\cdot):\sM\to(-\infty,+\infty]$ is lower semicontinuous.
  \item A sequence $s\e$ of measurable functions $s\e:\Omega\to\sM$ is called \textit{tight}, if there exists a normal integrand $h$ such that for every $\omega\in\Omega$ the function $h(\omega,\cdot)$ has compact sublevels in $\sM$ and $\limsup_{\eps\to 0}\int_\Omega h(\omega,s\e(\omega))\,dP(\omega)<\infty$.
  \item A Young measure in $\sM$ is a family $\boldsymbol{\mu}:=\left\{ \mu_{\omega}\right\} _{\omega\in\Omega}$
    of Borel probability measures on $\sM$ such that for all $B\in\mathcal B(\sM)$ the map $\Omega\ni \omega\mapsto \mu_\omega(B)\in\R$ is $\mathcal F$-measurable.
  \end{itemize}
\end{defn}
\begin{thm}[\mbox{\cite[Theorem I]{Balder1984}}]\label{thm:Balder} Let $s\e:\,\Omega\to\sM$ denote a tight sequence of measurable functions. Then there exists a subsequence, still indexed by $\eps$, and a Young measure ${\boldsymbol\mu}:\Omega\to\sM$ such that for every normal integrand $h:\,\Omega\times \sM\rightarrow(-\infty,+\infty]$ we have
\begin{equation}\label{eq:Balders-ineq}
\liminf_{\eps\to0}\int_{\Omega}h(\omega,s\e(\omega))\,dP(\omega)\geq\int_{\Omega}\int_{\sM}h(\omega,\xi) d\mu_{\omega}(\xi)dP(\omega)\,,
\end{equation}
provided that the negative part $h^-_\eps(\cdot)=|\min\{0,h(\cdot,s\e(\cdot))\}|$ is uniformly integrable.
Moreover, for $P$-a.a. $\omega\in\Omega_0$ the measure $\mu_\omega$ is supported in the set of all cluster points of $s\e(\omega)$, i.e.~in $\bigcup_{k=1}^\infty\overline{\{s\e(\omega)\,:\,\eps<\frac{1}{k}\}}$ (where the closure is taken in $\sM$).
\end{thm}
In order to apply the above theorem we require an appropriate  metric space in which two-scale convergent sequences and their limits embed:
\begin{lemma}\label{L:metric-struct}
  \begin{enumerate}[label=(\roman*)]
  \item   We denote by $\sM$ the set of all triples $(U,\eps,r)$ with $U\in\mbox{\rm Lin}(\sD)$, $\eps\geq 0$, $r\geq 0$.  $\sM$ endowed with the metric
  \begin{equation*}
    d((U_1,\eps_1,r_1),(U_2,\eps_2,r_2);\sM):=d(U_1,U_2;\mbox{\rm Lin}(\sD))+|\eps_1-\eps_2|+|r_1-r_2|
  \end{equation*}
  is a complete, separable metric space. 
\item For $\omega_0\in\Omega_0$ we denote by $\sM^{\omega_0}$ the set of all triples $(U,\eps,r)\in\sM$ such that
  \begin{equation}\label{eq:repr}
    U=
    \begin{cases}
      J^{\omega_0}_\eps u&\text{for some }u\in L^p(Q)\text{ with }\|u\|_{L^p(Q)}\leq r\text{ in the case }\eps>0,\\
      J_0 u&\text{for some }u\in\sB^p\text{ with }\|u\|_{\sB^p}\leq r\text{ in the case }\eps=0.
    \end{cases}
  \end{equation}
  Then $\sM^{\omega_0}$ is a closed subspace of $\sM$.
\item Let $\omega_0\in\Omega_0$, and $(U,\eps,r)\in\sM^{\omega_0}$. Then the function $u$ in the representation \eqref{eq:repr} of $U$ is unique, and
  \begin{equation}\label{eq:x6}
    \begin{cases}\displaystyle
      \|u\|_{L^p(Q)}=\sup\limits_{\varphi\in\overline{\sD},\ \|\varphi\|_{\sB^q}\leq 1}|U(\varphi)|&\text{if }\eps>0,\\
      \|u\|_{\sB^p}=\sup\limits_{\varphi\in \sD,\ \|\varphi\|_{\sB^q}\leq 1}|U(\varphi)|&\text{if }\eps=0.
    \end{cases}
  \end{equation}
\item For $\omega_0\in\Omega_0$ the function $\|\cdot\|_{\omega_0}:\sM^{\omega_0}\to[0,\infty)$, 
  \begin{equation*}
    \|(U,\eps,r)\|_{\omega_0}:=
    \begin{cases}
      \big(\sup\limits_{\varphi\in\overline{\sD},\ \|\varphi\|_{\sB^q}\leq 1}|U(\varphi)|^p+\eps+r^p\big)^{\frac{1}{p}}&\text{if }(U,\eps,r)\in\sM^{\omega_0},\,\eps>0,\\
      \big(\sup\limits_{\varphi\in \sD,\ \|\varphi\|_{\sB^q}\leq 1}|U(\varphi)|^p+r^p\big)^\frac1p&\text{if }(U,\eps,r)\in\sM^{\omega_0},\,\eps=0,\\
    \end{cases}
  \end{equation*}
  is lower semicontinuous on $\sM^{\omega_0}$. 
\item For all $(u,\eps)$ with $u\in L^p(Q)$ and $\eps>0$ we have $s:=(J^{\omega_0}_\eps u,\eps,\|u\|_{L^p(Q)})\in\sM^{\omega_0}$ and $\|s\|_{\omega_0}=\big(2\|u\|_{L^p(Q)}^p+\eps\big)^\frac1p$. Likewise, for all $(u,\eps)$ with $u\in\sB^p$ and $\eps=0$ we have $s=(J_0u,\eps,\|u\|_{\sB^p})$ and $\|s\|_{\omega_0}=2^\frac1p\|u\|_{\sB^p}$.
\item For all $R<\infty$ the set $\{(U,\eps,r)\in\sM^{\omega_0}\,:\,\|(U,\eps,r)\|_{\omega_0}\leq R\}$ is compact in $\sM$.
\item Let $\omega_0\in\Omega_0$ and let $u\e$ denote a bounded sequence in $L^p(Q)$. Then the triple $s\e:=(J^{\omega_0}\e u\e,\eps,\|u_{\varepsilon}\|_{L^p(Q)})$ defines a sequence in $\sM^{\omega_0}$. Moreover, we have $s\e\to s_0$ in $\sM$ as $\eps\to0$ if and only if $s_0=(J_0u,0,r)$ for some $u\in\sB^p$, $r\geq\|u\|_{\sB^p}$, and $u\e\tsq{\omega_0}u$.
  \end{enumerate}
\end{lemma}
\begin{proof}
  \begin{enumerate}[label=(\roman*)]
  \item This is a direct consequence of Lemma~\ref{L:metric-char} (i) and the fact that the product of complete, separable metric spaces remains complete and separable.
  \item Let $s_k:=(U_k,\eps_k,r_k)$ denote a sequence in $\sM^{\omega_0}$ that converges in $\sM$ to some $s_0=(U_0,\eps_0,r_0)$. We need to show that $s_0\in\sM^{\omega_0}$. By passing to a subsequence, it suffices to study the following three cases: $\eps_k>0$ for all $k\in\N_0$, $\eps_k=0$ for all $k\in\N_0$, and $\eps_0=0$ while $\eps_k>0$ for all $k\in\N$.

    Case 1: $\eps_k>0$ for all $k\in\N_0$.\\
    W.l.o.g.~we may assume that $\inf_k\eps_k>0$. Hence, there exist  $u_k\in L^p(Q)$ with $U_k=J^{\omega_0}_{\eps_k}u_k$. Since $r_k\to r$, we conclude that $(u_k)$ is bounded in $L^p(\Omega)$. We thus may pass to a subsequence (not relabeled) such that $u_k\weakto u_0$ weakly in $L^p(Q)$, and
    \begin{equation}\label{eq:x3}
      \|u_0\|_{L^p(Q)}\leq \liminf\limits_{k}\|u_k\|_{L^p(Q)}\leq \lim_k r_k=r_0.
    \end{equation}
    Moreover, $U_k\to U$ in the metric space $\mbox{\rm Lin}(\sD)$ implies pointwise convergence on $\sD$, and thus for all $\varphi_Q\in\sD_Q$ we have $U_k(\mathbf 1_{\Omega}\varphi_Q)=\int_Qu_k\varphi_Q\to \int_Qu_0\varphi_Q$. We thus conclude that $U_0(\mathbf 1_{\Omega}\varphi_Q)=\int_Q u_0\varphi_Q$. Since $\sD_Q\subset L^q(Q)$ dense, we deduce that $u_k\weakto u_0$ weakly in $L^p(Q)$ for the entire sequence.
    On the other hand the properties of the shift $\tau$ imply that for any $\varphi_\Omega\in\sD_\Omega$ we have $\varphi_\Omega(\tau_{\frac{\cdot}{\eps_k}}\omega_0)\to\varphi_\Omega(\tau_{\frac{\cdot}{\eps_0}}\omega_0)$ in $L^q(Q)$. Hence, for any $\varphi_\Omega\in\sD_\Omega$ and $\varphi_Q\in\sD_Q$ we have
    \begin{equation*}
      U_k(\varphi_\Omega\varphi_Q)=\int_Q u_k(x)\varphi_Q(x)\varphi_\Omega(\tau_{\frac{x}{\eps_k}}\omega_0)\,dx\to
      \int_Q u_0(x)\varphi_Q(x)\varphi_\Omega(\tau_{\frac{x}{\eps_0}}\omega_0)\,dx=J^{\omega_0}_{\eps_0}(\varphi_\Omega\varphi_Q)
    \end{equation*}
    and thus (by linearity) $U_0=J^{\omega_0}_{\eps_0}u_0$.

    Case 2: $\eps_k=0$ for all $k\in\N_0$.\\
    In this case there exist a bounded sequence $u_k$ in $\sB^p$ with $U_k=J_0u_k$ for $k\in\N$. By passing to a subsequence we may assume that $u_k\weakto u_0$ weakly in $\sB^p$ for some $u_0\in\sB^p$ with 
    \begin{equation}\label{eq:x4}
      \|u_0\|_{\sB^p}\leq \liminf\limits_{k}\|u_{\eps_k}\|_{\sB^p}\leq r_0.
    \end{equation}
    This implies that $U_k=J_0u_k\to J_0u_0$ in $\mbox{\rm Lin}(\sD)$. Hence, $U_0=J_0u_0$ and we conclude that $s_0\in\sM^{\omega_0}$.

    Case 3: $\eps_k>0$ for all $k\in\N$ and $\eps_0=0$.\\
    There exists a bounded sequence $u_k$ in $L^p(Q)$. Thanks to Lemma~\ref{lem:two-scale-limit}, by passing to a subsequence we may assume that $u_k\tsq{\omega_0} u_0$ for some $u\in \sB^p$ with 
    \begin{equation}\label{eq:x5}
      \|u_0\|_{\sB^p}\leq \liminf\limits_{k}\|u_k\|_{L^p(Q)}\leq r_0.
    \end{equation}
    Furthermore, Lemma~\ref{L:metric-char} implies that $J^{\omega_0}_{\eps_k}u_k\to J_0u_0$ in $\mbox{\rm Lin}(\sD)$, and thus $U_0=J_0u_0$. We conclude that $s_0\in\sM^{\omega_0}$.
\item We first argue that the representation \eqref{eq:repr} of $U$ by $u$ is unique. In the case $\eps>0$ suppose that $u,v\in L^p(Q)$ satisfy $U=J^{\omega_0}_{\eps}u=J^{\omega_0}_{\eps}v$. Then for all $\varphi_Q\in\sD_Q$ we have $\int_Q(u-v)\varphi_Q=J^{\omega_0}_\eps u(\mathbf 1_\Omega\varphi_Q)-J^{\omega_0}_\eps v(\mathbf 1_\Omega\varphi_Q)=U(\mathbf 1_\Omega\varphi_Q)-U(\mathbf 1_\Omega\varphi_Q)=0$, and since $\sD_Q\subset L^q(Q)$ dense, we conclude that $u=v$. In the case $\eps=0$ the statement follows by a similar argument from the fact that $\sD$ is dense $\sB^q$.
  To see \eqref{eq:x6} let $u$ and $U$ be related by \eqref{eq:repr}. Since $\overline\sD$ (resp. $\sD$) is dense in $L^q(Q)$ (resp. $\sB^q$), we have
  \begin{equation*}
    \begin{cases}
      \|u\|_{L^p(Q)}=\sup\limits_{\varphi\in\overline{\sD},\ \|\varphi\|_{\sB^q}\leq 1}|\int_Qu\varphi\,dx\,dP|=\sup\limits_{\varphi\in\overline{\sD},\ \|\varphi\|_{\sB^q}\leq 1}|U(\varphi)|&\text{if }\eps>0,\\
      \|u\|_{\sB^p}=\sup\limits_{\varphi\in\sD,\ \|\varphi\|_{\sB^q}\leq 1}|\int_{\Omega}\int_{Q}u\varphi\,dx\,dP|=\sup\limits_{\varphi\in\sD,\ \|\varphi\|_{\sB^q}\leq 1}|U(\varphi)|&\text{if }\eps=0.
    \end{cases}    
  \end{equation*}

  \item Let $s_k=(U_k,\eps_k,r_k)$ denote a sequence in $\sM^{\omega_0}$ that converges in $\sM$ to a limit $s_0=(U_0,\eps_0,r_0)$. By (ii) we have $s_0\in\sM^{\omega_0}$. For $k\in\N_0$ let $u_k$ in $L^p(Q)$ or $\sB^p$ denote the representation of $U_k$ in the sense of \eqref{eq:repr}. We may pass to a subsequence such that one of the three cases in (ii) applies and (as in (ii)) either $u_k$ weakly converges to $u_0$ (in $L^p(Q)$ or $\sB^p$), or $u_k\tsq{\omega_0}u_0$. In any of these cases the claimed lower semicontinuity of $\|\cdot\|_{\omega_0}$ follows from $\eps_k\to\eps_0$, $r_k\to r_0$, and \eqref{eq:x6} in connection with one of the lower semicontinuity estimates \eqref{eq:x3} -- \eqref{eq:x5}.
\item This follows from the definition and duality argument \eqref{eq:x6}.
\item Let $s_k$ denote a sequence in $\sM^{\omega_0}$. Let $u_k$ in $L^p(Q)$ or $\sB^p$ denote the (unique) representative of $U_k$ in the sense of \eqref{eq:repr}. Suppose that  $\|s_k\|_{\omega_0}\leq R$. Then  $(r_k)$ and $(\eps_k)$ are bounded sequences in $\R_{\geq 0}$, and $\sup_{k}\|u_k\|\leq \sup_kr_k<\infty$ (where $\|\cdot\|$ stands short for either $\|\cdot\|_{L^p(Q)}$ or $\|\cdot\|_{\sB^p}$). Thus we may pass to a subsequence such that~$r_k\to r_0$,  $\eps_k\to \eps_0$, and one of the following three cases applies:
  \begin{itemize}
  \item Case 1: $\inf_{k\in\N_0}\eps_k>0$. In that case we conclude (after passing to a further subsequence) that $u_k\weakto u_0$ weakly in $L^p(Q)$, and thus $U_k\to U_0=J^{\omega_0}_{\eps_0}u_0$ in $\mbox{Lin}(\sD)$.
  \item Case 2: $\eps_k=0$ for all $k\in\N_0$. In that case we conclude (after passing to a further subsequence) that $u_k\weakto u_0$ weakly in $\sB^p(Q)$, and thus $U_k\to U_0=J_0u_0$ in $\mbox{Lin}(\sD)$.
  \item Case 3: $\eps_k>0$ for all $k\in\N$ and $\eps_0=0$.  In that case we conclude (after passing to a further subsequence) that $u_k\tsq{\omega_0}u_0$, and thus $U_k\to U_0=J_0u_0$ in $\mbox{Lin}(\sD)$.
  \end{itemize}
  In all of these cases we deduce that $s_0=(U_0,\eps_0,r_0)\in\sM^{\omega_0}$, and $s_k\to s_0$ in $\sM$.
\item This is a direct consequence of (ii) -- (vi), and Lemma~\ref{L:metric-char}.
\end{enumerate}
\end{proof}

Now we are in position to prove Theorem~\ref{thm:Balder-Thm-two-scale}
\begin{proof}[Proof of Theorem~\ref{thm:Balder-Thm-two-scale}]
Let $\sM$, $\sM^{\omega_0}$, $J^{\omega_0}_\eps$ etc.~be defined as in Lemma~\ref{L:metric-struct}. 
\smallskip

{\it Step 1. (Identification of $(u\e)$ with a tight $\sM$-valued sequence).}
Since $u\e\in\sB^p$, by Fubini's theorem, we have $u\e(\omega,\cdot)\in L^p(Q)$ for $P$-a.a.~$\omega\in\Omega$. By modifying $u\e$ on a null-set in $\Omega\times Q$ (which does not alter two-scale limits in the mean), we may assume w.l.o.g.~that $u\e(\omega,\cdot)\in\ L^p(Q)$ for all $\omega\in\Omega$. Consider the measurable function $s\e:\Omega\to\sM$ defined as
\begin{equation*}
  s\e(\omega):=    \begin{cases}
      \big(J^{\omega}\e u\e(\omega,\cdot),\eps,\|u\e(\omega,\cdot)\|_{L^p(Q)}\big)&\text{if }\omega\in\Omega_0\\
      (0,0,0)&\text{else.}
    \end{cases}
\end{equation*}
We claim that $(s\e)$ is tight. To that end consider the integrand $h:\Omega\times\sM\to(-\infty,+\infty]$ defined by
\begin{equation*}
  h(\omega,(U,\eps,r)):=
  \begin{cases}
    \|(U,\eps,r)\|_{\omega}^p&\text{if }\omega\in\Omega_0\text{ and }(U,\eps,r)\in\sM^{\omega},\\
    +\infty&\text{else.}
  \end{cases}
\end{equation*}
From Lemma~\ref{L:metric-struct} (iv) and (vi) we deduce that $h$ is a normal integrand and $h(\omega,\cdot)$ has compact sublevels for all $\omega\in\Omega$. Moreover, for all $\omega_0\in\Omega_0$ we have $s\e(\omega_0)\in\sM^{\omega_0}$ and thus $h(\omega_0,s\e(\omega_0))=2\|u\e(\omega_0,\cdot)\|^p_{L^p(Q)}+\eps$. Hence,
\begin{equation*}
  \int_\Omega h(\omega,s\e(\omega))\,dP(\omega)=2\|u\e\|^p_{\sB^p}+\eps.
\end{equation*}
We conclude that $(s\e)$ is tight.

{\it Step 2. (Compactness and  definition of $\boldsymbol{\nu}$)}. By appealing to Theorem~\ref{thm:Balder} there exists a subsequence (still denoted by $\eps$) and a Young measure $\boldsymbol{\mu}$ that is generated by $(s\e)$. Let $\boldsymbol{\mu_1}$ denote the first component of $\boldsymbol{\mu}$, i.e.~the Young measure on $\mbox{\rm Lin}(\sD)$ characterized for $\omega\in\Omega$ by
\begin{equation*}
  \int_{\mbox{\rm Lin}(\sD)}f(\xi)\,d\mu_{1,\omega}(\xi)=\int_{\sM}f(\xi_1)\,d\mu_\omega(\xi),
\end{equation*}
for all $f:\mbox{\rm Lin}(\sD)\to\R$ continuous and bounded, where $\sM\ni\xi=(\xi_1,\xi_2,\xi_3)\to\xi_1\in\mbox{\rm Lin}(\sD)$ denotes the projection to the first component.
By Balder's theorem, $\mu_\omega$ is concentrated on the limit points of $(s\e(\omega))$. By Lemma~\ref{L:metric-struct} we deduce that for all $\omega\in\Omega_0$ any limit point $s_0(\omega)$ of $s\e(\omega)$ has the form $s_0(\omega)=(J_0u,0,r)$ where $0\leq r<\infty$ and $u\in\sB^p$ is a $\omega$-two-scale limit of a subsequence of $u\e(\omega,\cdot)$. Thus, $\mu_{1,\omega}$ is supported on $\{J_0u\,:\,u\in \CL(\omega,(u\e(\omega,\cdot))\}$ which in particular is a subset of $(\sB^q)^*$. Since $J_0:\sB^p\to (\sB^q)^*$ is an isometric isomorphism (by the Riesz-Frech\'et theorem), we conclude that $\boldsymbol{\nu}=\{\nu_\omega\}_{\omega\in\Omega}$, $\nu_\omega(B):=\mu_{1,\omega}(J_0B)$ (for all Borel sets $B\subset\sB^p$ where $\sB^p$ is equipped with the weak topology) defines a Young measure on $\sB^p$, and for all $\omega\in\Omega_0$, $\nu_\omega$ is supported on $\CL(\omega,(u\e(\omega,\cdot))$. 
\smallskip

{\it Step 3. (Lower semicontinuity estimate).} Note that $h:\Omega\times\sM\to[0,+\infty]$,
\begin{equation*}
  h(\omega,(U,\eps,r)):=
  \begin{cases}
    \sup_{\varphi\in\overline\sD,\,\|\varphi\|_{\sB^q}\leq 1}|U(\varphi)|^p&\text{if }\omega\in\Omega_0\text{ and }(U,\eps,r)\in\sM^{\omega},\eps>0,\\
    \sup_{\varphi\in\sD,\,\|\varphi\|_{\sB^q}\leq 1}|U(\varphi)|^p&\text{if }\omega\in\Omega_0\text{ and }(U,\eps,r)\in\sM^{\omega},\eps=0,\\
    +\infty&\text{else.}
  \end{cases}
\end{equation*}
defines a normal integrand, as can be seen as in the proof of Lemma~\ref{L:metric-struct}. Thus Theorem~\ref{thm:Balder} implies that
\begin{equation*}
  \liminf_{\eps\to 0}\int_\Omega h(\omega,s\e(\omega))\,dP(\omega)\geq \int_\Omega\int_{\sM}h(\omega,\xi)\,d\mu_\omega(\xi)dP(\omega).
\end{equation*}
In view of Lemma~\ref{L:metric-struct} we have $    \sup_{\varphi\in\overline\sD,\,\|\varphi\|_{\sB^q}\leq 1}|(J^\omega_\eps u_\eps)(\omega,\cdot))(\varphi)|=\|u\e(\omega,\cdot)\|_{L^p(Q)}$ for $\omega\in\Omega_0$, and thus the left-hand side turns into $\liminf_{\eps\to 0}\|u\e\|^p_{\sB^p}$. Thanks to the definition of $\boldsymbol{\nu}$ the right-hand side turns into $\int_\Omega \int_{\sB^p}\|v\|_{\sB^p}^p\,d\nu_\omega(v)dP(\omega)$.
\smallskip

{\it Step 4. (Identification of the two-scale limit in the mean)}.
Let $\varphi\in\sD_0$. Then $h:\Omega\times\sM\to[0,+\infty]$,
\begin{equation*}
  h(\omega,(U,\eps,r)):=
  \begin{cases}
    U(\varphi)&\text{if }\omega\in\Omega_0,\,(U,\eps,r)\in\sM^\omega,\\
    +\infty&\text{else.}
  \end{cases}
\end{equation*}
defines a normal integrand. Since $h(\omega,s\e(\omega))=\int_Qu\e(\omega,x)\unf^*\varphi(\omega,x)\,dx$ for $P$-a.a.~$\omega\in\Omega$, we deduce that $|h(\cdot, s\e(\cdot))|$ is uniformly integrable. Thus, \eqref{eq:Balders-ineq} applied to $\pm h$ and the definition of $\boldsymbol{\nu}$ imply that
\begin{eqnarray*}
  \lim\limits_{\eps\to 0}\int_\Omega\int_Qu\e(\omega,x)(\unf^*\varphi)(\omega,x)\,dx\,dP(\omega)&=&  \lim\limits_{\eps\to 0}\int_\Omega h(\omega,s\e(\omega))\,dP(\omega)\\
  &=&\int_\Omega\int_{\sB^p}h(\omega,v)\,d\nu_\omega(v)\,dP(\omega)\\
  &=&\int_\Omega\int_{\sB^p}\ex{\int_Qv\varphi}\,d\nu_\omega(v)\,dP(\omega).
\end{eqnarray*}
Set $u:=\int_\Omega\int_{\sB^p}v\,d\nu_\omega(v)dP(\omega)\in\sB^p$. Then Fubini's theorem yields
\begin{eqnarray*}
  \lim\limits_{\eps\to 0}\int_\Omega\int_Qu\e(\omega,x)(\unf^*\varphi)(\omega,x)\,dx\,dP(\omega)&=& \ex{\int_Qu\varphi}.
\end{eqnarray*}
Since $\mbox{span}(\sD_0)\subset \sB^q$ dense, we conclude that $u\e\ts u$.
\smallskip

{\it Step 5. Recovery of quenched two-scale convergence}. Suppose that $\nu_\omega$ is a delta distribution on $\sB^p$, say $\nu_\omega=\delta_{v(\omega)}$ for some measurable $v:\Omega\to\sB^p$. Note that $h:\Omega\times\sM\to[0,+\infty]$,
\begin{equation*}
  h(\omega,(U,\eps,r)):=-d(U,J_0v(\omega);\mbox{\rm Lin}(\sD))
\end{equation*}
is a normal integrand and $|h(\cdot, s\e(\cdot))|$ is uniformly integrable. Thus,  \eqref{eq:Balders-ineq} yields
\begin{eqnarray*}
  &&\limsup\limits_{\eps\to 0}\int_{\Omega} d(J^\omega\e u\e(\omega,\cdot),J_0v(\omega);\mbox{\rm Lin}(\sD))\,dP(\omega)\\
  &=&-\liminf\limits_{\eps\to 0}\int_\Omega h(\omega,s\e(\omega))\,dP(\omega)\\
  &\leq&-\int_\Omega\int_{\sB^p}h(\omega,J_0v)\,d\nu_\omega(v)\,dP(\omega)=-\int_\Omega h(\omega,J_0v(\omega))\,dP(\omega)=0.
\end{eqnarray*}
Thus, there exists a subsequence (not relabeled) such that $d(J^\omega\e u\e(\omega,\cdot),J_0v(\omega);\mbox{\rm Lin}(\sD))\to 0$ for a.a.~$\omega\in\Omega_0$. In view of Lemma~\ref{L:metric-char} this implies that $u\e\tsq{\omega}v(\omega)$ for a.a.~$\omega\in\Omega_0$.
\end{proof}

\begin{proof}[Proof of Lemma~\ref{lem:Balder-Lem-two-scale}]
  {\it Step 1. Representation of the functional by a lower semicontinuous integrand on $\sM$.}\\
  For all $\omega_0\in\Omega_0$ and $s=(U,\eps,r)\in\sM^{\omega_0}$ we write $\pi^{\omega_0}(s)$ for the unique representation $u$ in $\sB^p$ (resp. $L^p(Q)$) of $U$ in the sense of \eqref{eq:repr}. We thus may define for $\omega\in\Omega_0$ and $s\in\sM^{\omega_0}$ the integrand
\begin{equation*}
  \overline h(\omega_0,s):=
  \begin{cases}
    \int_Qh(\tau_{\frac{x}{\eps}}\omega,x,(\pi^{\omega_0}s)(x))\,dx&\text{if }s=(U,\eps,s)\text{ with }\eps>0,\\
    \int_\Omega\int_Qh(\omega,x,(\pi^{\omega_0}s)(x))\,dx\,dP(\omega)&\text{if }s=(U,\eps,s)\text{ with }\eps=0.
  \end{cases}
\end{equation*}

We extend $\overline h(\omega_0,\cdot)$ to $\sM$ by $+\infty$, and define $\overline h(\omega,\cdot)\equiv 0$ for $\omega\in\Omega\setminus\Omega_0$. We claim that $\overline h(\omega,\cdot):\sM\to(-\infty,+\infty]$ is lower semicontinuous for all $\omega\in\Omega$. It suffices to consider $\omega_0\in\Omega_0$ and a convergent sequence $s_k=(U_k,\eps_k,r_k)$ in $\sM^{\omega_0}$. For brevity we only consider the (interesting) case when $\eps_k\downarrow \eps_0=0$. Set $u_k:=\pi^{\omega_0}(s_k)$. By construction we have
\begin{equation*}
  \overline h(\omega_0,s_k)=\int_Q h(\tau_{\frac{x}{\eps_k}}\omega_0,u_k(\omega_0,x))\,dx,
\end{equation*}
and
\begin{equation*}
  \overline h(\omega_0,s_0)=\int_{\Omega}\int_Q h(\omega,x,u_0(\omega,x))\,dx\,dP(\omega).
\end{equation*}
Since $s_k\to s_0$ and $\eps_k\to 0$, Lemma~\ref{L:metric-struct} (vi) implies that $u_k\tsq{\omega_0}u_0$, and since $h$ satisfies \ref{ass:lsc} from Remark~\ref{lem:General-Hom-Convex}, we conclude that $\liminf\limits_{k}  \overline h(\omega_0,s_k)\geq   \overline h(\omega_0,s_0)$, and thus $\overline h$ is a normal integrand.

{\it Step 2. Conclusion.}\\
As in Step~1 of the proof of Theorem~\ref{thm:Balder-Thm-two-scale} we may associate with the sequence $(u\e)$ a sequence of measurable functions $s\e:\Omega\to\sM$ that (after passing to a subsequence that we do not relabel) generates a Young measure $\boldsymbol{\mu}$ on $\sM$. Since by assumption $u\e$ generates the Young measure ${\boldsymbol \nu}$ on $\sB^p$, we deduce that the first component $\boldsymbol{\mu_1}$ satisfies $\nu_\omega(B)=\mu_{\omega}(J_0B)$ for any Borel set $B$. Applying \eqref{eq:Balders-ineq} to the integrand $\overline h$ of Step~1, yields
\begin{align*}
  \liminf\limits_{\e\to 0}\int_\Omega\int_Q &h(\tau_{\frac{x}{\eps}}\omega_0,u\e(\omega_0,x))\,dx\,dP(\omega)\\
  &=\liminf\limits_{\e\to 0}\int_\Omega\overline h(\omega,s\e(\omega))\,dP(\omega)\\
  &\geq\int_\Omega\int_{\sM}\overline h(\omega,\xi)\,d\mu_\omega(\xi)\,dP(\omega)\\
  &=\int_\Omega\int_{\sB^p}\Big(\int_\Omega\int_Qh(\tomega,x,v(\tomega,x))\,dx\,dP(\tomega)\Big)\,d\nu_\omega(v)\,dP(\omega).
\end{align*}
\end{proof}

\begin{proof}[Proof of Lemma~\ref{L:fromquenchedtomean}]
  By (b) and (c) the sequence $(\tilde u\e)$ is bounded in $\sB^p$ and thus we can pass to a subsequence such that $(\tilde u\e)$ generates a Young measure $\boldsymbol \nu$. Set $\tilde u:=\int_\Omega\int_{\sB^p}v\,d\nu_\omega(v)\,dP(\omega)$ and note that Theorem~\ref{thm:Balder-Thm-two-scale} implies that $\tilde u\e\wt \tilde u$ weakly two-scale in the mean. On the other hand the theorem implies that $\nu_\omega$ concentrates on the quenched two-scale cluster points of $(u^\omega\e)$ (for a.a.~$\omega\in\Omega$). Hence, in view of (a) we conclude that for a.a.~$\omega\in\Omega$ the measure $\nu_\omega$ is a Dirac measure concentrated on $u$, and thus $\tilde u=u$ a.e.~in $\Omega\times Q$.  
\end{proof}

\section{Convex homogenization via stochastic unfolding}\label{Section_Applications}

In this section we revisit a standard model example of stochastic homogenization of integral functionals from the viewpoint of stochastic two-scale convergence and unfolding. In particular, we discuss two examples of convex homogenization problems that can be treated with stochastic two-scale convergence in the \textit{mean}, but not with the \textit{quenched} variant. In the first example in Section \ref{sec:nonerg} the randomness is \textit{nonergodic} and thus quenched two-scale convergence does not apply.  In the second example, in Section \ref{sec:1012}, we consider a \textit{variance-regularization} to treat a convex minimization problem with degenerate growth conditions. In these two examples we also demonstrate the simplicity of using the stochastic unfolding operator. Furthermore, in Section \ref{Section:4:3} we use the results of Section \ref{Section_4} to further reveal the structure of the previously obtained limits in the classical ergodic case with non-degenerate growth with help of Young measures. In particular, we show how to lift mean homogenization results to quenched statements.
\subsection{Nonergodic case}\label{sec:nonerg}
In this section we consider a nonergodic stationary medium. Such random ensembles arise naturally, e.g., in the context of periodic representative volume element (RVE) approximations, see \cite{fischer2019optimal}. For example, we may consider a family of i.i.d. random variables $\cb{\overline{\omega}(z)}_{z\in \mathbb{Z}^d}$. A realization of a stationary and ergodic random checkerboard is given by
\begin{equation*}
\omega: \R^d \to \R, \quad \omega(x) = \sum_{i\in \mathbb{Z}^d} \mathbf{1}_{i+y+\Box}(x)\overline{\omega}(\lfloor x \rfloor),
\end{equation*}
where $\lfloor x\rfloor\in \mathbb{Z}^d$ is the integer part of $x$ and $y\in \Box$ is the center of the checkerboard chosen uniformly from $\Box=[0,1)^d$. For $L\in \N$, we may consider the map $\pi_{L}: \omega \mapsto \omega_{L}$ given by $\pi_{L}\omega(x)=\omega(x)$ for $x\in [0,L)^d$ and $\pi_{L}\omega$ is $L$-periodically extended. The push forward of the map $\pi_{L}$ defines a stationary and nonergodic probability measure, that is a starting point in the periodic RVE method. Another standard example of a nonergodic structure may be obtained by considering a medium with a \textit{noncompatible} quasiperiodic microstructure, see \cite[Example 1.2]{Zhikov2006}.
\bigskip

In this section we consider the following situation.
Let $p\in (1,\infty)$ and $Q\subset \R^d$ be open and bounded. We consider $V:\Omega\times Q\times \re{d}\rightarrow \re{}$ and assume:
\begin{enumerate}[label = (A\arabic*)]
\item \label{ass:1} $V(\cdot,\cdot, F)$ is $\mathcal{F}\otimes \mathcal{L}(Q)$-measurable for all $F\in \R^{d}$.
\item \label{ass:2} $V(\omega, x, \cdot)$ is convex for a.a. $(\omega,x)\in \Omega\times Q$.
\item \label{ass:3} There exists $C>0$ such that
\begin{equation*}
\frac{1}{C}|F|^p-C\leq V(\omega, x, F) \leq C(|F|^p+1)
\end{equation*}
for a.a. $(\omega,x) \in \Omega\times Q$ and all $F\in \re{d}$.
\end{enumerate}
We consider the functional 
\begin{equation}\label{energy}
 \mathcal{E}_{\varepsilon}: L^p(\Omega)\otimes \sob_0(Q)\rightarrow \re{},\quad
 \mathcal{E}_{\varepsilon}(u)=\ex{\int_Q V(\tau_{\frac{x}{\varepsilon}}\omega, x,\nabla u(\omega,x))dx}.
\end{equation}

Under assumptions \ref{ass:1}-\ref{ass:3}, in the limit $\varepsilon\rightarrow 0$ we obtain the two-scale functional
\begin{align}\label{energy_hom}
\begin{split}
& \mathcal{E}_0:\brac{L^p_{{\inv}}(\Omega)\otimes \sob_0(Q)} \times \brac{L^p_{\pot}(\Omega)\otimes L^p(Q)} \to\R,\\
& \mathcal{E}_0(u,\chi)=\ex{\int_Q V(\omega, x, \nabla u(\omega,x)+ \chi(\omega,x)) dx}.
\end{split}
\end{align}
\begin{thm}[Two-scale homogenization]\label{thm1}
Let $p\in (1,\infty)$ and $Q\subset \R^d$ be open and bounded. Assume \ref{ass:1}-\ref{ass:3}.
\begin{enumerate}[label=(\roman*)]
\item (Compactness and liminf inequality.) Let $u\e \in L^p(\Omega)\otimes \sob_0(Q)$ be such that
$\limsup_{\varepsilon\rightarrow 0}\mathcal{E}_{\varepsilon}(u\e)<\infty$. 
There exist $(u,\chi) \in \brac{L^p_{{\inv}}(\Omega)\otimes \sob_0(Q)} \times \brac{L^p_{\pot}(\Omega)\otimes L^p(Q)}$ and a subsequence (not relabeled) such that
\begin{align}
& u\e \wt u \text{ in }\ltp, \quad \nabla u\e \wt \nabla u+\chi  \text{ in }\ltp, \label{convergence} \\ & \liminf_{\varepsilon\to 0}\mathcal{E}\e(u\e)\geq \mathcal{E}_0(u,\chi). \label{eq:838}
\end{align}
\item (Limsup inequality.) Let $(u,\chi)\in \brac{L^p_{{\inv}}(\Omega)\otimes \sob_0(Q)} \times \brac{L^p_{\pot}(\Omega)\otimes L^p(Q)}$. There exists a sequence $u\e \in L^p(\Omega)\otimes \sob_0(Q)$ such that
\begin{align}
& u\e \st u \text{ in }\ltp, \quad  \nabla u\e \st \nabla u+\chi  \text{ in }\ltp, \label{eq:842}\\
& \limsup_{\varepsilon\rightarrow 0}\mathcal{E}_{\varepsilon}(u\e) \leq \mathcal{E}_0(u,\chi).\label{eq:843}
\end{align}
\end{enumerate}
\end{thm}
\begin{proof}[Proof of Theorem \ref{thm1}]
(i)
The Poincar{\'e} inequality and \ref{ass:3} imply that $u\e$ is bounded in $L^p(\Omega)\otimes W^{1,p}(Q)$. By Proposition \ref{prop:292} (ii) there exist $u \in L^p_{{\inv}}(\Omega)\otimes W^{1,p}(Q)$ and $\chi \in L^p_{\pot}(\Omega)\otimes L^p(Q)$ such that \eqref{convergence} holds. Also, note that $P_{\mathrm{inv}}u\e \rightharpoonup u$ weakly in $L^p(\Omega)\otimes W^{1,p}(Q)$ and $P_{\mathrm{inv}}u\e \in L^p_{\mathrm{inv}}(\Omega)\otimes W^{1,p}_0(Q)$, which implies that $u$ also has $0$ boundary values, i.e., $u \in L^p_{\mathrm{inv}}(\Omega)\otimes W^{1,p}_0(Q)$. Finally, we note that, see \cite[Proposition 3.5 (i)]{heida2019stochastic},
\begin{equation}\label{eq:930:2}
\ex{\int_{Q} V(\tau_{\frac{x}{\varepsilon}}\omega,x,v(\omega,x))} = \ex{\int_{Q} V(\omega,x,\unf v(\omega,x))} \quad \text{for any }v \in L^p(\Omega\times Q),
\end{equation}
and thus using the convexity of $V$ we conclude
\begin{equation*}
\liminf_{\varepsilon\to 0} \mathcal{E}\e(u\e) = \liminf_{\varepsilon\to 0} \ex{\int_{Q}V(\omega,x,\unf \nabla u\e)}\geq \mathcal{E}_0(u,\chi).
\end{equation*}
\medskip

(ii) The existence of a sequence $u\e$ with \eqref{eq:842} follows from Proposition \ref{prop:292} (iii). Furthermore, \eqref{eq:930:2} and the growth assumption \ref{ass:3} yield
\begin{equation*}
\lim_{\varepsilon\to 0}\mathcal{E}_{\varepsilon}(u\e) = \lim_{\varepsilon\to 0} \ex{\int_{Q}V(\omega,x,\unf \nabla u\e)} = \mathcal{E}_{0}(u,\chi). 
\end{equation*}
This concludes the claim, in particular, we even show a stronger result stating convergence of the energy. 
\end{proof}

\begin{remark}[Convergence of minimizers]\label{C:thm1}
We consider the setting of Theorem \ref{thm1}.  Let $u\e\in L^p(\Omega)\otimes W^{1,p}_0(Q)$ be a minimizer of the functional 
\begin{equation*}
\mathcal{I}_{\varepsilon}:L^p(\Omega)\otimes W^{1,p}_0(Q)\to \R, \quad \mathcal{I}_{\varepsilon}(u)= \cE\e(u) - \ex{\int_{Q}u\e f\e dx},
\end{equation*}
where $f\e \in L^q(\Omega\times Q)$ and $f\e \overset{2}{\to}f$ with $f\in L^q(Q)$ and $\frac{1}{p}+\frac{1}{q}=1$. By a standard argument from the theory  of $\Gamma$-convergence Theorem \ref{thm1} (cf. \cite[Corollary 7.2]{varga2019stochastic}) implies
that there exist a subsequence (not relabeled),  $u\in L^p_{\inv}(\Omega)\times W^{1,p}_0(Q)$, and $\chi\in L^p_{\pot}(\Omega)\otimes L^p(Q)$ such that $u\e \wt u \text{ in }\ltp$, $\nabla u\e \wt \nabla u+\chi  \text{ in }\ltp$, and 
  \begin{equation*}
    \lim\limits_{\eps\to 0}\min\mathcal{I}\e=    \lim\limits_{\eps\to 0}\mathcal{I}\e(u\e)=\mathcal{I}_0(u,\chi)=\min\mathcal{I}_0,
  \end{equation*}
where $\mathcal{I}_0: L^p_{\mathrm{inv}}(\Omega)\otimes W^{1,p}_0(Q)\to \R$ is given by $\mathcal{I}_{0}(u)= \mathcal{E}_0(u)-\int_{Q}f u dx$. This, in particular, rigorously justifies the formal two-scale expansion $\nabla u\e(x) \approx \nabla u_{0}(\omega, x) + \chi( \tau_{\frac{x}{\varepsilon}}\omega, x)$. 
\end{remark}

\begin{remark}[Uniqueness]\label{remark:1099}
If $V(\omega, x,\cdot)$ is strictly convex the minimizers are unique and the convergence in the above remark holds for the entire sequence.
\end{remark}
\subsection{Variance-regularization applied to degenerate growth }\label{sec:1012}
In this section we consider homogenization of convex functionals with degenerate growth. More precisely, we consider an integrand $V$ that satisfies \ref{ass:1}, \ref{ass:2} and the following assumption (as a replacement of \ref{ass:3}):
\begin{enumerate}[label = (A\arabic*'), start=3]
\item \label{ass:3b} There exists $C>0$ and a random variable $\lambda\in L^1(\Omega)$ such that
  \begin{equation}
    \label{eq:moment}
    \ex{\lambda^{-\frac{1}{p-1}}}^{p-1}<C
  \end{equation}
and
\begin{equation*}
\lambda(\omega)|F|^p-C\leq V(\omega, x, F) \leq C(\lambda(\omega)|F|^p+1)
\end{equation*}
for a.a. $(\omega,x) \in \Omega\times Q$ and all $F\in \re{d}$.
\end{enumerate}
Moreover, we assume that $\ex{\cdot}$ is ergodic. For $\varepsilon>0$ we consider the following functional
\begin{equation*}
\mathcal{E}_{\varepsilon}:L^1(\Omega \times Q) \to \R\cup\cb{\infty}, \qquad \mathcal{E}_{\varepsilon}(u)= \ex{\int_{Q} V(\tau_{\frac{x}{\varepsilon}}\omega, x, \nabla u) dx},
\end{equation*}
for $u \in X_{\varepsilon}$ and $\mathcal{E}_{\varepsilon}(u) = \infty$ otherwise. Here $X_{\varepsilon}$ denotes the closure of $\cb{u\in L^p(\Omega)\otimes W^{1,p}_0(Q)}$ w.r.t.~the weighted norm
\begin{equation*}
  \|u\|_{\lambda_{\varepsilon}} := \ex{\int_{Q}\lambda(\tau_{\frac{x}{\varepsilon}}\omega) |\nabla u|^p dx}^{\frac{1}{p}}.
\end{equation*}
Recently, in \cite{NSS, HoppePhd, HNS} it shown that $\mathcal{E}_\varepsilon$ Mosco-converges to the functional
\begin{equation*}
  \mathcal E_{\hom}: L^1(Q)\to\R\cup \cb{\infty},\qquad \mathcal E_{\hom}(u):=\int_Q V_{\hom}(x,\nabla u(x))\,dx,
\end{equation*}
for $u \in W^{1,p}_0(Q)$ and $\mathcal E_{\hom}(u)=\infty$ otherwise, where $V_{\hom}:Q\times\R^{d}\to\R$ is given by the homogenization formula,
\begin{align}\label{equation2}
  V_{\hom}(x,F)=\inf_{\chi\in L^p_{\pot}(\Omega)}\ex{V(\omega, x,F+\chi(\omega))},
\end{align}
for $x\in Q$ and $F\in \R^{d}$. Moreover, it is shown that $V_{\mathrm{hom}}$ is a normal convex integrand that satisfies a standard $p$-growth condition. Note that the assumption \ref{ass:3b} in comparison to \ref{ass:3} makes a genuine difference in regard to the homogenization formula \eqref{equation2}. In particular, in the setting of assumption \ref{ass:3} minimizers are attained due to the coercivity of the underlying functional in $L^p_{\pot}(\Omega)$. It is thus easy to see that the homogenized integrand satisfies $p$-growth condition as well, see Section \ref{Section:4:3} below.  
On the other hand, in the setting of this section assuming \ref{ass:3b}, \eqref{equation2} is a degenerate minimization problem and a priori minimizers will only have finite first moments. An additional argument is required to infer that $V_{\hom}$ in \eqref{equation2} is non-degenerate, in particular, in \cite[Theorem 3.1]{NSS} it is shown that there exists a constant $C'>0$ such that for all $x\in Q$ and $F\in \R^d$ we have
\begin{equation}\label{eq:984:5}
\frac{1}{C'} |F|^p - C' \leq V_{\mathrm{hom}}(x,F) \leq C'\brac{|F|^p+1}.
\end{equation}
One of the difficulties in the proof of the homogenization result for $\mathcal E_{\varepsilon}$ is due to the fact that the domain of the functionals are $\varepsilon$-dependent. Moreover, assumption \ref{ass:3b} only yields equicoercivity in $L^1(\Omega)\otimes W^{1,1}_0(Q)$, while the limit $\mathcal E_{\hom}$ is properly defined on $W^{1,p}(Q)$. Therefore, in practice it is convenient to regularize the problem: For $\delta>0$ we consider the regularized homogenization formula
\begin{equation*}
  V_{\hom,\delta}(x,F)=\inf_{\chi\in L^p_{\pot}(\Omega)}\ex{V(\omega, x,F+\chi(\omega))+\delta|\chi(\omega)|^p}.
\end{equation*}
It is simple to show that the infimum on the right-hand side is attained by a unique minimizer. We also consider the corresponding regularized homogenized integral functional
  \begin{equation*}
    \mathcal E_{\hom,\delta}:L^1(Q)\to\R\cup \cb{\infty},\qquad \mathcal E_{\hom,\delta}(u):=\int_QV_{\hom,\delta}(\nabla u)\,dx,
  \end{equation*}
for $u\in W^{1,p}_0(Q)$ and $\mathcal{E}_{\hom, \delta}(u)=\infty$ otherwise. Furthermore, thanks to \ref{ass:3b}, it is relatively easy to see that this regularization is consistent:
\begin{lemma}\label{lemma:1054} Let $p\in (1,\infty)$ and $Q\subset \R^d$ be open and bounded. Assume \ref{ass:1}, \ref{ass:2} and \ref{ass:3b}. Then, for all $x\in Q$ and $F\in\R^{d}$, we have 
\begin{equation}\label{eq:901}
\lim_{\delta\to 0} V_{\hom,\delta}(x,F) = V_{\hom}(x,F).
\end{equation} 
Moreover, $\mathcal{E}_{\hom, \delta}$ Mosco converges to $\mathcal{E}_{\hom}$ as $\delta \to 0$, i.e., the following statements hold:
\begin{enumerate}[label=(\roman*)]
\item If $u_{\delta}\rightharpoonup u$ weakly in $L^1(Q)$, then 
\begin{equation*}
\liminf_{\delta \to 0}\mathcal{E}_{\mathrm{hom},\delta}(u_{\delta}) \geq \mathcal{E}_{\mathrm{hom}}(u).
\end{equation*}
\item For any $u \in L^1(Q)$ there exists a sequence $u_{\delta}\in L^1(Q)$ such that
\begin{equation*}
u_{\delta} \to u \quad \text{strongly in }L^1(Q), \quad \mathcal{E}_{\mathrm{hom},\delta}(u_{\delta})\to \mathcal{E}_{\mathrm{hom}}(u).
\end{equation*}
\end{enumerate}

\end{lemma}
\begin{proof}
Let $F\in \R^{d}$ and $x\in Q$. Since $\delta>0$, we have $V_{\hom,\delta}(x,F)\geq V_{\hom}(x,F)$. On the other hand, we consider a minimizing sequence $\chi_{\eta}\in L^p_{\mathrm{pot}}(\Omega)$ in \eqref{equation2}, e.g., 
\begin{equation*}
\ex{V(\omega, x, F+ \chi_{\eta})} \leq V_{\hom}(x,F) + \eta. 
\end{equation*}
We have
\begin{equation*}
V_{\hom,\delta}(x,F)\leq \ex{V(\omega, x, F+ \chi_{\eta})+ \delta |\chi_{\eta}|^p} \leq V_{\hom}(x,F)+ \eta + \delta \ex{|\chi_{\eta}|^p}.
\end{equation*}
Letting first $\delta\to 0$ and then $\eta\to 0$, we conclude \eqref{eq:901}.

We further consider a sequence $u_{\delta}$ such that $u_{\delta}\rightharpoonup u$ weakly in $L^1(Q)$ as $\delta \to 0$. We assume without loss of generality that $\limsup_{\delta \to 0}\mathcal{E}_{\hom,\delta}(u_{\delta})< \infty$. This, in particular, with the help of \eqref{eq:984:5} and the Poincar{\'e} inequality implies that $\limsup_{\delta \to 0}  \|u_{\delta}\|_{W^{1,p}_0(Q)}< \infty$. Thus, up to a subsequence, we have $ u_{\delta} \rightharpoonup  u$ weakly in $W^{1,p}_0(Q)$. Using this, we obtain
\begin{equation*}
\liminf_{\delta \to 0}\mathcal{E}_{\hom, \delta}(u_{\delta})\geq  \liminf_{\delta \to 0}\mathcal{E}_{\hom}(u_{\delta}) \geq \mathcal{E}_{\hom}(u).
\end{equation*}
The first inequality follows by \eqref{eq:901} and the second is a consequence of the fact that $V_{\hom}(x,\cdot)$ is convex and of Fatou's Lemma. We conclude that (i) holds.

If $u \notin \mathrm{dom}(\mathcal{E}_{\hom})$, we simply choose $u_{\delta}=u$. On the other hand, for $u\in \mathrm{dom}(\mathcal{E}_{\hom})=W^{1,p}_0(Q)$, \eqref{eq:901} and the dominated convergence theorem yield 
\begin{equation*}
\lim_{\delta\to 0}\mathcal{E}_{\hom, \delta}(u) =\mathcal{E}_{\hom}(u). 
\end{equation*}
This means that (ii) holds. 
\end{proof}
In the following we introduce a variance regularization of the original functional $\mathcal E\e$ that removes the degeneracy of the problem and thus can be analyzed by the standard strategy of Section~\ref{sec:nonerg}. For $\delta>0$, we consider
\begin{equation}\label{energy-delta}
 \mathcal{E}_{\varepsilon,\delta}:L^1(\Omega \times Q)\rightarrow \re{},\quad
 \mathcal{E}_{\varepsilon,\delta}(u)=\ex{\int_Q V(\tau_{\frac{x}{\varepsilon}}\omega, x,\nabla u(x))+\delta|\nabla u(x)-\ex{\nabla u( x)}|^pdx},
\end{equation}
for $u\in L^p(\Omega)\otimes W^{1,p}_0(Q)$ and $\mathcal{E}_{\varepsilon, \delta}= \infty$ otherwise.
Due to the structure of the additional term, we call it a variance-regularization and we note that it only becomes active for non-deterministic functions. For fixed $\delta>0$, the functional $\mathcal E_{\varepsilon,\delta}$ is equicoercive on $L^p(\Omega)\otimes W^{1,p}_0(Q)$:
\begin{lemma}\label{lemma:940} Let $p\in (1,\infty)$ and $Q\subset \R^d$ be open and bounded. Assume \ref{ass:1} and \ref{ass:3b}. Then there exists $C=C(Q,p)>0$ such that, for all $u\in L^p(\Omega)\otimes W^{1,p}_0(Q)$, it holds
  \begin{equation*}
    \ex{\int_Q|\nabla u|}^p+\delta\ex{\int_Q|\nabla u|^p}\leq C\big(\mathcal E_{\varepsilon,\delta}(u)+1\big).
  \end{equation*}
\end{lemma}
\begin{proof}
By Jensen's and H\"older's inequalities we have
\begin{equation*}
  \ex{\int_Q|\nabla u|dx}^p \leq |Q|^{p-1}\int_Q\ex{|\nabla u|}^p\leq|Q|^{p-1}\ex{{\lambda_{\varepsilon}^{-\frac{1}{p-1}}}}^{p-1}\,\ex{\int_Q\lambda_{\varepsilon}|\nabla u|^p},
\end{equation*}
where we use the notation $\lambda_{\varepsilon}(x,\omega)=\lambda(\tau_{\frac{x}{\varepsilon}}\omega)$. Furthermore, using \ref{ass:3b}, we conclude that
\begin{equation*}
\ex{\int_Q|\nabla u|dx}^p \leq C(Q,p)\brac{\mathcal{E}_{\varepsilon,\delta}(u) + 1}.
\end{equation*}
In the end, using the variance-regularization we obtain
\begin{eqnarray*}
  2^{-p}\ex{\int_Q|\nabla u|^p}&\leq&\ex{\int_Q|\nabla u-\ex{\nabla u}|^p}+\int_Q\ex{|\nabla u|}^p\\
  &\leq& \frac{C}{\delta}\brac{\mathcal E_{\varepsilon,\delta}(u)+1}+C\big(\mathcal E_{\varepsilon,\delta}(u)+1\big).
\end{eqnarray*}
This concludes the proof.  
\end{proof}

The regularization on the $\varepsilon$-level is also consistent. In particular, we show that in the limit $\delta \to 0$, we recover $\mathcal{E}_{\varepsilon}$.
We discuss the mean functionals $\mathcal{E}_{\varepsilon,\delta}$ and $\mathcal{E}_{\varepsilon}$, since the former does not admit a well-defined pointwise evaluation in $\omega$ for the reason of the nonlocal variance term. Also, for the same reason the quenched version of stochastic two-scale convergence is not suitable for this setting and we apply the unfolding procedure. On the other hand, the homogenization of $\mathcal{E}_{\varepsilon}$ can be conducted on the level of typical realizations, that was in fact studied in \cite{NSS, HoppePhd, HNS}.

\begin{lemma}\label{lemma:1124} Let $p\in (1,\infty)$ and $Q\subset \R^d$ be open and bounded. Assume \ref{ass:1}, \ref{ass:2} and \ref{ass:3b}. Then, $\mathcal{E}_{\varepsilon, \delta}$ Mosco converges to $\mathcal{E}_{\varepsilon}$ as $\delta \to 0$ i.e., the following statements hold:
\begin{enumerate}[label=(\roman*)]
\item If $u_{\delta}\rightharpoonup u$ weakly in $L^1(\Omega \times Q)$, then 
\begin{equation*}
\liminf_{\delta \to 0}\mathcal{E}_{\varepsilon, \delta}(u_{\delta}) \geq \mathcal{E}_{\varepsilon}(u).
\end{equation*}
\item For any $u \in L^1(\Omega \times Q)$ there exists a sequence $u_{\delta}\in L^1(\Omega \times Q)$ such that
\begin{equation*}
u_{\delta} \to u \quad \text{strongly in }L^1(\Omega \times Q), \quad \mathcal{E}_{\varepsilon,\delta}(u_{\delta})\to \mathcal{E}_{\varepsilon}(u).
\end{equation*}
\end{enumerate}

\end{lemma}

\begin{proof}
(i) Let $u_{\delta}$ be a sequence such that $u_{\delta}\rightharpoonup u$ weakly in $L^1(\Omega \times Q)$. Without loss of generality we assume that $\limsup_{\delta \to 0}\mathcal{E}_{\varepsilon, \delta}(u_{\delta}) <\infty$. This and the proof of Lemma \ref{lemma:940} imply that the sequence $\lambda_{\varepsilon}^{\frac{1}{p}} \nabla u_{\delta}$ is bounded in $L^p(\Omega \times Q)$ with the notation $\lambda_{\varepsilon}(x,\omega)=\lambda(\tau_{\frac{x}{\varepsilon}}\omega)$. This means that, up to a subsequence, we have $\lambda_{\varepsilon}^{\frac{1}{p}} \nabla u_{\delta} \rightharpoonup \psi$ weakly in $L^p(\Omega \times Q)$ for some $\psi \in L^p(\Omega \times Q)$. Thus, for an arbitrary $\eta \in L^{\infty}(\Omega\times Q)$, we have
\begin{equation*}
\ex{\int_{Q}\nabla u_{\delta} \eta dx} = \ex{\int_{Q}\lambda_{\varepsilon}^{\frac{1}{p}}\nabla u_{\delta} \lambda_{\varepsilon}^{-\frac{1}{p}}\eta dx}\to \ex{\int_{Q}\psi \lambda_{\varepsilon}^{-\frac{1}{p}}\eta dx} \quad \text{as }\varepsilon \to 0.
\end{equation*} 
This means that $\nabla u_{\delta}$ converges weakly in $L^1(\Omega \times Q)$ and since $u_{\delta}\rightharpoonup u$ weakly in $L^1(\Omega\times Q)$ we may conclude that $\nabla u_{\delta}\rightharpoonup \nabla u$ weakly in $L^1(\Omega\times Q)$.
This yields
\begin{equation*}
\liminf_{\delta \to 0}\mathcal{E}_{\varepsilon, \delta}(u_{\delta}) \geq \liminf_{\delta \to 0} \mathcal{E}_{\varepsilon}(u_{\delta}) \geq \mathcal{E}_{\varepsilon}(u). 
\end{equation*}
(ii) For an arbitrary $u\in \mathrm{dom}(\mathcal{E}_{\varepsilon})\subset X_{\varepsilon}$, we find a sequence $u_{\eta}\in L^p(\Omega)\otimes W^{1,p}_0(Q)$ such that, for $\eta \to 0$, 
\begin{equation*}
u_{\eta} \to u \quad \text{strongly in }L^1(\Omega)\otimes W^{1,1}_0(Q), \quad \ex{\int_{Q}\lambda_{\varepsilon} |\nabla u_{\eta}-\nabla u|^p dx} \to 0.
\end{equation*}
Using this and the dominated convergence theorem, we conclude that
\begin{equation*}
\lim_{\eta \to 0}\mathcal{E}_{\varepsilon}(u_{\eta}) =\mathcal{E}_{\varepsilon}(u).   
\end{equation*}
This in turn yields
\begin{equation*}
\limsup_{\eta\to 0} \limsup_{\delta\to 0} |\mathcal{E}_{\varepsilon, \delta}(u_{\eta})- \mathcal{E}_{\varepsilon}(u)| = 0.
\end{equation*}
We extract a diagonal sequence $\eta(\delta)\to 0$ as $\delta \to 0$ such that $u_{\delta}:=u_{\eta(\delta)}$ satisfies $u_{\delta}\to u$ strongly in $L^1(\Omega \times Q)$ and $\mathcal{E}_{\varepsilon,\delta}(u_{\delta})\to \mathcal{E}_{\varepsilon}(u)$. This concludes the proof.
\end{proof}

The homogenization of the regularized functional $\mathcal{E}_{\varepsilon, \delta}$ boils down to a very similar simple argumentation as in Section \ref{sec:nonerg}. 

\begin{thm}\label{theorem:1149} Let $p\in (1,\infty)$ and $Q\subset \R^d$ be open and bounded. Assume \ref{ass:1}, \ref{ass:2} and \ref{ass:3b}. For all $\delta>0$, as $\varepsilon \to 0$,  $\mathcal{E}_{\varepsilon, \delta}$ Mosco converges to $\mathcal{E}_{\mathrm{hom}, \delta}$ in the following sense:
\begin{enumerate}[label=(\roman*)]
\item  Let $u\e \in L^p(\Omega)\otimes \sob_0(Q)$ be such that
$\limsup_{\varepsilon\rightarrow 0}\mathcal{E}_{\varepsilon, \delta}(u\e)<\infty$. Then there exist $(u,\chi) \in \sob_0(Q) \times \brac{L^p_{\pot}(\Omega)\otimes L^p(Q)}$ and a subsequence (not relabeled) such that
\begin{equation*}
u\e \wt u \text{ in }\ltp, \quad \nabla u\e \wt \nabla u+\chi  \text{ in }\ltp. 
\end{equation*}

\item If $u\e \in L^1(\Omega \times Q)$, $u \in L^1(Q)$ and $\unf u_{\varepsilon} \rightharpoonup u$ weakly in $L^1(\Omega\times Q)$, then
\begin{equation*}
\liminf_{\varepsilon\to 0}\mathcal{E}_{\varepsilon, \delta}(u\e) \geq \mathcal{E}_{\hom,\delta}(u).
\end{equation*}

\item For any $u \in L^1(Q)$, there exists a sequence $u\e \in L^1(\Omega \times Q)$ such that
\begin{equation*}
\unf u\e \to u \quad \text{strongly in }L^1(\Omega\times Q), \quad \mathcal{E}_{\varepsilon,\delta}(u\e)\to \mathcal{E}_{\hom, \delta}(u).
\end{equation*}
\end{enumerate}
\end{thm}

\begin{proof}
(i) The statement follows analogously to the proof of Theorem \ref{thm1} (i).

(ii) Let $\unf u_{\varepsilon}\to u$ weakly in $L^1(\Omega\times Q)$. We may assume without loss of generality that $\limsup_{\varepsilon \to 0}\mathcal{E}_{\varepsilon,\delta}(u_{\varepsilon})<\infty$. In this case, Lemma \ref{lemma:940} implies that $u\e$ is bounded in $L^p(\Omega)\otimes W^{1,p}_0(Q)$. We may proceed analogously to Theorem \ref{thm1} and Remark \ref{remark:1099} to obtain
\begin{equation*}
\liminf_{\varepsilon\to 0}\mathcal{E}_{\varepsilon,\delta}(u_{\varepsilon}) \geq \mathcal{E}_{\hom, \delta}(u).
\end{equation*}
(ii) This part is analogous to Theorem \ref{thm1} and Remark \ref{remark:1099}.
\end{proof}
The results of Lemmas \eqref{lemma:1054} and \eqref{lemma:1124}, Theorem \eqref{theorem:1149} and \cite{NSS,HoppePhd, HNS} can be summarized in the following commutative diagram:
\begin{equation*}
  \begin{array}{ccc}
   \qquad\quad \mathcal E_{\varepsilon,\delta} & \stackrel{(\delta \to 0)}{\to}& \mathcal E_{\varepsilon}\qquad \quad \\
    \scaleto{(\varepsilon \to 0)}{8pt} \downarrow   &  &\downarrow \scaleto{(\varepsilon \to 0)}{8pt}\\
    \qquad\quad \mathcal E_{\hom,\delta}&\stackrel{(\delta \to 0)}{\to}&\mathcal E_{\hom}\qquad \quad
  \end{array}
\end{equation*}
The arrows denote Mosco convergence in the corresponding convergence regimes.
\subsection{Quenched homogenization of convex functionals}\label{Section:4:3}
In this section we demonstrate how to lift homogenization results w.r.t.~two-scale convergence in the mean to quenched statements at the example of Section \ref{sec:nonerg}. Throughout this section we assume that $\ex{\cdot}$ is ergodic. For $\omega\in\Omega$ we define $\cE^{\omega}\e: W^{1,p}_0(Q)\to \re{}$, 
\begin{equation*}
\cE^{\omega}_{\eps}(u):=\int_{Q}V\left(\tau_{\frac{x}{\eps}}\omega, x,\nabla u(x)\right)\,dx,
\end{equation*}
with $V$ satisfying \ref{ass:1}-\ref{ass:3}. The goal of this section is to relate two-scale limits of ``mean''-minimizers, i.e.~functions $u\e\in L^p(\Omega)\otimes W^{1,p}_0(Q)$ that minimize $\cE_{\eps}$, with limits of ``quenched''-minimizers, i.e.~families  $\{u\e(\omega)\}_{\omega\in\Omega}$ of minimizers to $\cE^{\omega}\e$ in $W^{1,p}_0(Q)$. We also remark that if $V(\omega,x,\cdot)$ is strictly convex $u\e$ and $\cb{u\e(\omega)}_{\omega\in\Omega}$ may be identified since minimizers of both functionals $\mathcal{E}\e$ and $\mathcal{E}\e^{\omega}$ are unique. 

Before presenting the main result of this section, we remark that in the ergodic case, the limit functional \eqref{energy_hom} reduces to a single-scale energy
\begin{equation*}
 \mathcal{E}_{\hom}:\sob_0(Q) \rightarrow \re{}, \quad
 \mathcal{E}_{\hom}(u)=\int_Q V_{\hom}(x,\nabla u(x))dx,
\end{equation*}
where the homogenized integrand $V_{\hom}$ is given for $x\in \R^d$ and $F\in \R^{d}$ by 
\begin{align}\label{equation}
V_{\hom}(x,F)=\inf_{\chi\in L^p_{\pot}(\Omega)}\ex{V(\omega, x,F+\chi(\omega))}.
\end{align}
In particular, we may obtain an analogous statement to Theorem \ref{thm1} where we replace $\mathcal{E}_0$ with $\mathcal{E}_{\mathrm{hom}}$. The proof of this follows analogously with the only difference that in the construction of the recovery sequence we first need to find $\chi$ such that $\mathcal{E}_0(u,\chi)=\mathcal{E}_{\mathrm{hom}}(u)$. This is done by a usual measurable selection argument, cf. \cite[Theorem 7.6]{varga2019stochastic}.  

\begin{thm}
  \label{thm:Quenched-hom-convex-grad} Let $p\in (1,\infty)$, $Q\subset \R^d$ be open and bounded, and $\ex{\cdot}$ be ergodic. Assume \ref{ass:1}-\ref{ass:3}. Let $u\e\in L^{p}(\Omega)\otimes W_{0}^{1,p}(Q)$
  be a minimizer of $\cE_{\eps}$. Then there exists a subsequence such that $(u\e,\nabla u\e)$ generates a Young measure $\boldsymbol{\nu}$ in $\sB:=(\sB^p)^{1+d}$ in the sense of Theorem~\ref{thm:Balder-Thm-two-scale}, and for $P$-a.a.~$\omega\in\Omega$, $\nu_{\omega}$ concentrates on the set $    \big\{\,(u,\nabla u+\chi)\,:\,\cE_0(u,\chi)=\min\cE_0\,\big\}$ of minimizers of the limit functional.
  Moreover, if  $V(\omega,x,\cdot)$ is strictly convex for all $x\in Q$ and $P$-a.a.~$\omega\in\Omega$, then the minimizer $u\e$ of $\cE_{\eps}$ and the minimizer $(u,\chi)$ of $\cE_0$ are  unique, and for $P$-a.a.~$\omega\in\Omega$ we have (for a not relabeled subsequence)
  \begin{gather*}
    u\e(\omega,\cdot)\weakto u\text{ weakly in }W^{1,p}(Q),\qquad u\e(\omega,\cdot)\tsq{\omega}u,\qquad\nabla u\e(\omega,\cdot)\tsq{\omega}\nabla u+\chi,\\
    \text{and }\min\cE^\omega_\eps=\cE^\omega_\eps(u\e(\omega,\cdot))\to \cE_0(u,\chi)=\min\cE_0.
  \end{gather*}
\end{thm}
\begin{remark}[Identification of quenched two-scale cluster points]
  If we combine Theorem~\ref{thm:Quenched-hom-convex-grad} with the identification of the support of the Young measure in Theorem~\ref{thm:Balder-Thm-two-scale} we conclude the following: There exists a subsequence such that
$(u\e,\nabla u\e)$ two-scale converges in the mean to a limit of the form $(u_0,\nabla u_0+\chi_0)$ with $\cE_0(u_0,\chi_0)=\min\cE_0$, and for a.a.~$\omega\in\Omega$ the set of quenched $\omega$-two-scale cluster points $\CL(\omega, (u\e(\omega,\cdot),\nabla u\e(\omega,\cdot)))$ is contained in $\big\{\,(u,\nabla u+\chi)\,:\,\cE_0(u,\chi)=\min\cE_0\,\big\}$. In the strictly convex case we further obtain that $\CL(\omega, (u\e(\omega,\cdot),\nabla u\e(\omega,\cdot)))=\{(u,\nabla u+\chi)\}$ where $(u,\chi)$ is the unique minimizer to $\cE_0$. Note, however, that our argument (that extracts quenched two-scale limits from the sequence of ``mean'' minimizers) involves an exceptional $P$-null-set that a priori depends on the selected subsequence. This is in contrast to the classical result in \cite{DalMaso1986} which is based on a subadditive ergodic theorem and states that there exists a set of full measure $\Omega'$ such that for all $\omega\in\Omega'$ the minimizer $u\e^\omega$ to $\cE^{\omega}_\eps$ weakly converges in $W^{1,p}(Q)$ to the deterministic minimizer $u$ of the reduced functional $\cE_{\hom}$ for any sequence $\eps\to 0$.
\end{remark}
In the proof of Theorem~\ref{thm:Quenched-hom-convex-grad} we combine homogenization in the mean in form of Theorem~\ref{thm1}, the connection to quenched two-scale limits via Young measures in form of Theorem~\ref{thm:Balder-Thm-two-scale}, and a recent result described in Remark~\ref{lem:General-Hom-Convex} by Nesenenko and the first author.
\begin{proof}[Proof of Theorem~\ref{thm:Quenched-hom-convex-grad}]  
  {\it Step 1. (Identification of the support of $\boldsymbol{\nu}$).}
  
  Since $u\e$ is a sequence of minimizers, by Corollary~\ref{C:thm1} there exists a subsequence (not relabeled) and minimizers $(u,\chi)\in W^{1,p}_0(Q)\times (L^p_{\pot}(\Omega)\otimes L^p(Q))$ of $\cE_0$ such that
  that $u\e \wt u \text{ in }\ltp$, $\nabla u\e \wt \nabla u+\chi  \text{ in }\ltp^{d}$, and 
  \begin{equation}\label{eq:conv-minima}
    \lim\limits_{\eps\to 0}\min\cE\e=    \lim\limits_{\eps\to 0}\cE\e(u\e)=\cE_0(u,\chi)=\min\cE_0.
  \end{equation}
  In particular, the sequence $(u\e,\nabla u\e)$ is bounded in $\sB$.  By Theorem~\ref{thm:Balder-Thm-two-scale} we may pass to a further subsequence (not relabeled) such that $(u\e,\nabla u\e)$ generates a Young measure $\boldsymbol{\nu}$ on $\sB$.   Since $\nu_\omega$ is supported on the set of quenched $\omega$-two-scale cluster points of $(u\e(\omega,\cdot),\nabla u\e(\omega,\cdot))$, we deduce from Lemma~\ref{lem:sto-conver-grad} that the support of $\nu_\omega$ is contained in $\sB_0:=\{\xi=(\xi_1,\xi_2)=(u',\nabla u'+\chi')\,:\,u'\in W^{1,p}_0(Q),\,\chi\in L^p_{\pot}(\Omega)\otimes L^p(Q)\}$ which is a closed subspace of $\sB$. Moreover, thanks to the relation of the generated Young measure and stochastic two-scale convergence in the mean, we have $(u,\chi)=\int_\Omega \int_{\sB_0}(\xi_1,\xi_2-\nabla\xi_1)\,\nu_\omega(d\xi)\,dP(\omega)$. Furthermore, Lemma~\ref{lem:Balder-Lem-two-scale} implies that
  \begin{equation*}
    \lim\limits_{\eps\to 0}\cE_\eps(u\e)\geq \int_\Omega\int_{\sB}\Big(\int_\Omega\int_Q V(\tomega,x,\xi_2)\,dx\,dP(\tomega)\Big)\,\nu_\omega(d\xi)\,dP(\omega).
  \end{equation*}
  In view of \eqref{eq:conv-minima} and the fact that $\nu_\omega$ is supported in $\sB_0$, we conclude that
  \begin{equation*}
    \min\cE_0\geq \int_\Omega\int_{\sB_0}\cE_0(\xi_1,\xi_2-\nabla\xi_1)\,\nu_\omega(d\xi)\,dP(\omega)\geq \min\cE_0\int_\Omega\int_{\sB_0}\nu_\omega(d\xi)dP(\omega).
  \end{equation*}
  Since $\int_\Omega\int_{\sB_0}\nu_\omega(d\xi)dP(\omega)=1$, we have $\int_\Omega\int_{\sB_0}|\cE_0(\xi_1,\xi_2-\nabla\xi_1)-\min\cE_0|\,\nu_\omega(d\xi)\,dP(\omega)= 0$, and thus we conclude that for $P$-a.a.~$\omega\in\Omega_0$,  $\nu_\omega$ concentrates on $\{(u,\nabla u+\chi)\,:\,\cE_0(u,\chi)=\min\cE_0\}$.
  \smallskip

  {\it Step 2. (The strictly convex case).}

  The uniqueness of $u\e$ and $(u,\chi)$ is clear. From Step~1 we thus conclude that $\nu_\omega=\delta_{\xi}$ where $\xi=(u,\nabla u+\chi)$. Theorem~\ref{thm:Balder-Thm-two-scale} implies that $(u\e(\omega,\cdot),\nabla u\e(\omega,\cdot))\tsq{\omega}(u,\nabla u+\chi)$ (for $P$-a.a.~$\omega\in\Omega$). 
  By Lemma~\ref{lem:Balder-Lem-two-scale} we have for $P$-a.a.~$\omega\in\Omega$,
  \begin{equation*}
    \liminf\limits_{\eps\to 0}\cE^\omega_\eps(u\e(\omega,\cdot))\geq \cE_0(u,\chi)=\min\cE_0.
  \end{equation*}
  On the other hand, since $u\e(\omega,\cdot)$ minimizes $\cE^\omega_\eps$, we deduce by a standard argument that for $P$-a.a.~$\omega\in\Omega$,
  \begin{equation*}
    \lim\limits_{\eps\to 0}\min\cE^\omega_\eps=\lim\limits_{\eps\to 0}\cE^\omega_\eps(u\e(\omega,\cdot))=\cE_0(u,\chi)=\min\cE_0.
  \end{equation*}
\end{proof}

\section*{Acknowledgments}

The authors thank Alexander Mielke for fruitful discussions and valuable comments. MH has been funded by Deutsche Forschungsgemeinschaft (DFG) through grant CRC 1114 ``Scaling Cascades
in Complex Systems'', Project C05 ``Effective models for materials and interfaces with multiple scales''. SN and MV acknowledge funding by the Deutsche Forschungsgemeinschaft (DFG, German
Research Foundation) – project number 405009441.

\bibliographystyle{abbrv}
\bibliography{ref}

\begin{thebibliography}{10}

\bibitem{allaire1992homogenization}
G.~Allaire.
\newblock Homogenization and two-scale convergence.
\newblock {\em SIAM Journal on Mathematical Analysis}, 23(6):1482--1518, 1992.

\bibitem{andrews1998stochastic}
K.~T. Andrews and S.~Wright.
\newblock Stochastic homogenization of elliptic boundary-value problems with
  {$L^p$}-data.
\newblock {\em Asymptotic Analysis}, 17(3):165--184, 1998.

\bibitem{arbogast1990derivation}
T.~Arbogast, J.~Douglas, Jr, and U.~Hornung.
\newblock Derivation of the double porosity model of single phase flow via
  homogenization theory.
\newblock {\em SIAM Journal on Mathematical Analysis}, 21(4):823--836, 1990.

\bibitem{Balder1984}
E.~J. Balder.
\newblock A general approach to lower semicontinuity and lower closure in
  optimal control theory.
\newblock {\em SIAM Journal on Control and Optimization}, 22(4):570--598, 1984.

\bibitem{bourgeat1994rigorous}
A.~Bourgeat, S.~Luckhaus, and A.~{Mikeli{\'c}}.
\newblock A rigorous result for a double porosity model of immiscible two-phase
  flow.
\newblock {\em Comptes Rendusa l’Acad{\'e}mie des Sciences}, 320:1289--1294,
  1994.

\bibitem{bourgeat1994stochastic}
A.~Bourgeat, A.~Mikeli{\'c}, and S.~Wright.
\newblock Stochastic two-scale convergence in the mean and applications.
\newblock {\em J. reine angew. Math}, 456(1):19--51, 1994.

\bibitem{cioranescu2004homogenization}
D.~Cioranescu, A.~Damlamian, and R.~De~Arcangelis.
\newblock Homogenization of nonlinear integrals via the periodic unfolding
  method.
\newblock {\em Comptes Rendus Mathematique}, 339(1):77--82, 2004.

\bibitem{cioranescu2012periodic}
D.~Cioranescu, A.~Damlamian, P.~Donato, G.~Griso, and R.~Zaki.
\newblock The periodic unfolding method in domains with holes.
\newblock {\em SIAM Journal on Mathematical Analysis}, 44(2):718--760, 2012.

\bibitem{cioranescu2002periodic}
D.~Cioranescu, A.~Damlamian, and G.~Griso.
\newblock Periodic unfolding and homogenization.
\newblock {\em Comptes Rendus Mathematique}, 335(1):99--104, 2002.

\bibitem{cioranescu2008periodic}
D.~Cioranescu, A.~Damlamian, and G.~Griso.
\newblock The periodic unfolding method in homogenization.
\newblock {\em SIAM Journal on Mathematical Analysis}, 40(4):1585--1620, 2008.

\bibitem{DalMaso1986}
G.~Dal~Maso and L.~Modica.
\newblock Nonlinear stochastic homogenization.
\newblock {\em Annali di matematica pura ed applicata}, 144(1):347--389, 1986.

\bibitem{Daley1988}
D.~Daley and D.~Vere-Jones.
\newblock An introduction to the theory of point processes, 1988.

\bibitem{fischer2019optimal}
J.~Fischer and S.~Neukamm.
\newblock Optimal homogenization rates in stochastic homogenization of
  nonlinear uniformly elliptic equations and systems.
\newblock {\em arXiv preprint arXiv:1908.02273}, 2019.

\bibitem{griso2004error}
G.~Griso.
\newblock Error estimate and unfolding for periodic homogenization.
\newblock {\em Asymptotic Analysis}, 40(3, 4):269--286, 2004.

\bibitem{hanke2011homogenization}
H.~Hanke.
\newblock Homogenization in gradient plasticity.
\newblock {\em Mathematical Models and Methods in Applied Sciences},
  21(08):1651--1684, 2011.

\bibitem{heida2011extension}
M.~Heida.
\newblock An extension of the stochastic two-scale convergence method and
  application.
\newblock {\em Asymptotic Analysis}, 72(1-2):1--30, 2011.

\bibitem{Heida2017b}
M.~Heida.
\newblock Stochastic homogenization of rate-independent systems and
  applications.
\newblock {\em Continuum Mechanics and Thermodynamics}, 29(3):853--894, 2017.

\bibitem{HeidaNesenenko2017monotone}
M.~Heida and S.~Nesenenko.
\newblock Stochastic homogenization of rate-dependent models of monotone type
  in plasticity.
\newblock {\em Asymptotic Analysis}, 112(3-4):185--212, 2019.

\bibitem{heida2019stochastic}
M.~Heida, S.~Neukamm, and M.~Varga.
\newblock Stochastic homogenization of {$\Lambda$}-convex gradient flows.
\newblock {\em Discrete \& Continuous Dynamical Systems – S}, 2020.

\bibitem{HoppePhd}
H.~Hoppe.
\newblock {\em Homogenization of Rapidly Oscillating Riemannian Manifolds}.
\newblock Dissertation, TU Dresden, 2020.
\newblock \url{https://nbn-resolving.org/urn:nbn:de:bsz:14-qucosa2-743766}.

\bibitem{HNS}
H.~Hoppe, S.~Neukamm, and M.~Sch{\"a}ffner.
\newblock Stochastic homogenization of non-convex integral functionals with
  degenerate growth.
\newblock {\em (in preparation)}, 2021.

\bibitem{jikov2012homogenization}
V.~V. Jikov, S.~M. Kozlov, and O.~A. Oleinik.
\newblock {\em Homogenization of differential operators and integral
  functionals}.
\newblock Springer Science \& Business Media, 2012.

\bibitem{Kozlov1979}
S.~M. Kozlov.
\newblock Averaging of random operators.
\newblock {\em Matematicheskii Sbornik}, 151(2):188--202, 1979.

\bibitem{liero2015homogenization}
M.~Liero and S.~Reichelt.
\newblock Homogenization of {C}ahn--{H}illiard-type equations via evolutionary
  {$\Gamma$}-convergence.
\newblock {\em Nonlinear Differential Equations and Applications NoDEA},
  25(1):6, 2018.

\bibitem{lukkassen2002two}
D.~Lukkassen, G.~Nguetseng, and P.~Wall.
\newblock Two-scale convergence.
\newblock {\em International Journal of Pure and Applied Mathematics},
  2(1):35--86, 2002.

\bibitem{mielke2014two}
A.~Mielke, S.~Reichelt, and M.~Thomas.
\newblock Two-scale homogenization of nonlinear reaction-diffusion systems with
  slow diffusion.
\newblock {\em Networks \& Heterogeneous Media}, 9(2), 2014.

\bibitem{mielke2007two}
A.~Mielke and A.~M. Timofte.
\newblock Two-scale homogenization for evolutionary variational inequalities
  via the energetic formulation.
\newblock {\em SIAM Journal on Mathematical Analysis}, 39(2):642--668, 2007.

\bibitem{neukamm2010homogenization}
S.~Neukamm.
\newblock Homogenization, linearization and dimension reduction in elasticity
  with variational methods.
\newblock {\em Technische Universit{\"a}t M{\"u}nchen}, 2010.

\bibitem{NSS}
S.~Neukamm, M.~Sch{\"a}ffner, and A.~Schl{\"o}merkemper.
\newblock Stochastic homogenization of nonconvex discrete energies with
  degenerate growth.
\newblock {\em SIAM Journal on Mathematical Analysis}, 49(3):1761--1809, 2017.

\bibitem{neukamm2017stochastic}
S.~Neukamm and M.~Varga.
\newblock Stochastic unfolding and homogenization of spring network models.
\newblock {\em Multiscale Modeling \& Simulation}, 16(2):857--899, 2018.

\bibitem{NeukammVargaWaurick}
S.~Neukamm, M.~Varga, and M.~Waurick.
\newblock Two-scale homogenization of abstract linear time-dependent {PDEs}.
\newblock {\em Asymptotic Analysis}, (Pre-press):1--41, 2020.

\bibitem{nguetseng1989general}
G.~Nguetseng.
\newblock A general convergence result for a functional related to the theory
  of homogenization.
\newblock {\em SIAM Journal on Mathematical Analysis}, 20(3):608--623, 1989.

\bibitem{Papanicolaou1979}
G.~C. Papanicolaou and S.~S. Varadhan.
\newblock Boundary value problems with rapidly oscillating random coefficients.
\newblock {\em Random fields}, 1:835--873, 1979.

\bibitem{varga2019stochastic}
M.~Varga.
\newblock {\em Stochastic unfolding and homogenization of evolutionary gradient
  systems}.
\newblock Dissertation, TU Dresden, 2019.
\newblock \url{https://nbn-resolving.org/urn:nbn:de:bsz:14-qucosa2-349342}.

\bibitem{Visintin2006}
A.~Visintin.
\newblock Towards a two-scale calculus.
\newblock {\em ESAIM: Control, Optimisation and Calculus of Variations},
  12(3):371--397, 2006.

\bibitem{vogt1980}
C.~Vogt.
\newblock A homogenization theorem leading to a {Volterra-integrodifferential}
  equation for permeation chromotography.
\newblock {\em Preprint No 155, SFB 123, Heidelberg}, 1982.

\bibitem{Zhikov2000}
V.~V. Zhikov.
\newblock On an extension of the method of two-scale convergence and its
  applications.
\newblock {\em Sbornik: Mathematics}, 191(7):973, 2000.

\bibitem{Zhikov2006}
V.~V. Zhikov and A.~Pyatnitskii.
\newblock Homogenization of random singular structures and random measures.
\newblock {\em Izvestiya: Mathematics}, 70(1):19--67, 2006.

\end{thebibliography}

\end{document}